\documentclass{amsart}

\usepackage{amsthm,amssymb,amsmath,a4wide,graphicx,tikz,bm}
\usepackage[english]{babel}
\usepackage{subfigure}
\usepackage{forest}
\usepackage{stmaryrd}
\definecolor{antiquefuchsia}{rgb}{0.57, 0.36, 0.51}
\definecolor{bittersweet}{rgb}{1.0, 0.44, 0.37}
\definecolor{asparagus}{rgb}{0.53, 0.66, 0.42}

\title{Constructions of normal numbers with infinitely many digits}
\author[Aafko Boonstra] {Aafko Boonstra$^\dagger$}
\author[Charlene Kalle]{Charlene Kalle$^\ddagger$}

\address[$\dagger$]{Department of Mathematics, VU University Amsterdam, De Boelelaan 1111, 1081HV Amsterdam, The Netherlands}
\email[Aafko Boonstra]{a.j.w.boonstra@vu.nl}
\address[$\ddagger$]{Mathematisch Instituut, Leiden University, Niels Bohrweg 1, 2333CA Leiden, The Netherlands}
\email[Charlene Kalle]{kallecccj@math.leidenuniv.nl}

\subjclass[2020]{11K16, 11A63, 05C05}
\keywords{normal numbers, digit frequency, L\"uroth expansions, GLS expansions}

\begin{document}

\newtheorem{prop}{Proposition}[section]
\newtheorem{theorem}{Theorem}[section]
\newtheorem{lemma}{Lemma}[section]
\newtheorem{cor}{Corollary}[section]
\newtheorem{remark}{Remark}[section]
\theoremstyle{definition}
\newtheorem{defn}{Definition}[section]
\newtheorem{ex}{Example}[section]
\tikzstyle{line} = [draw, -latex']

\begin{abstract}
Let $L=(L_d)_{d \in \mathbb N}$ be any ordered probability sequence, i.e., satisfying $0 < L_{d+1} \le L_d$ for each $d \in \mathbb N$ and $\sum_{d \in \mathbb N} L_d =1$. We construct sequences $A = (a_i)_{i \in \mathbb N}$ on the countably infinite alphabet $\mathbb N$ in which each possible block of digits $\alpha_1, \ldots, \alpha_k \in \mathbb N$, $k \in \mathbb N$, occurs with frequency $\prod_{d=1}^k L_{\alpha_d}$. In other words, we construct $L$-normal sequences. These sequences can then be projected to normal numbers in various affine number systems, such as real numbers $x \in [0,1]$ that are normal in GLS number systems that correspond to the sequence $L$ or higher dimensional variants. In particular, this construction provides a family of numbers that have a normal L\"uroth expansion.
\end{abstract}
\maketitle

\section{Introduction}
For any integer $N \ge 2$ each number $x \in [0,1]$ has a base $N$-expansion:
\[ x = \sum_{i \ge 1} \frac{c_i}{N^i}, \quad c_i \in \{0,1, \ldots, N-1\}, \, i \in \mathbb N.\]
The corresponding digit set is the set $\{0,1, \ldots, N-1\}$. A number $x \in [0,1]$ is called {\em normal} in base $N \ge 2$ if for each $k \ge 1$ and any digits $\gamma_1, \ldots, \gamma_k \in \{0,1,2, \ldots, N-1\}$ it holds that
\[ \lim_{n \to \infty} \frac{\# \{ 1 \le i \le n \, : \, c_i = \gamma_1, \ldots, c_{i+k-1}=\gamma_k\}}{n} = \frac1{N^k},\]
i.e., if each block $\gamma_1, \ldots, \gamma_k$ of $k$ digits occurs with frequency $\frac1{N^k}$ in the base $N$-expansion of $x$. It is a consequence of Borel's normal number theorem from 1909 \cite{Bor09} that Lebesgue almost all numbers $x \in [0,1]$ are normal in base $N$ for any integer $N \ge 2$. It is a different matter, however, to find explicit examples of such numbers. One of the most famous examples of a normal number in integer base $N \ge 2$ is Champernowne's constant obtained by concatenating the base $N$ representations of the natural numbers in increasing order, i.e.,
\[ x = 0.123\cdots (N-1) 10 11 12 \cdots 1(N-1) 202122 \cdots 2(N-1)\cdots,\]
evaluated in base $N$, see \cite{Cha33}. The literature on normal numbers in integer base expansions is  extensive. For some background and further information, we refer to e.g.~\cite{KN74,Bug12} and the references therein.

\medskip
Normality in base $N$ is essentially a property of the digit sequence $(c_i)_{i\in \mathbb N}$ with respect to the probability vector $(\frac1N, \ldots, \frac1N)$ assigning frequency $\frac1N$ to each of the digits, together with some independence on the occurrence of digits. Normal numbers are then obtained by projecting a sequence $(c_i)_{i \in \mathbb N}$ in which all blocks of digits appear with the right frequencies to the corresponding point $ \sum_{k \ge 1} \frac{c_k}{N^k} \in [0,1]$. This article is concerned with normality in case the digit set contains a countably infinite number of digits, i.e., is given by $\mathbb N$, and the frequencies of the digits are given by any ordered probability sequence $L = (L_d)_{d \in \mathbb N}$, so with $0 < L_{d+1} \le L_d$ for each $d \in \mathbb N$ and $\sum_{d \in \mathbb N}L_d =1$. In analogy to the base $N$-expansions, we call a sequence $(a_i)_{i \in \mathbb N} \in \mathbb N^\mathbb N$ {\em $L$-normal} if for each $k \in \mathbb N$ and each $\alpha_1, \ldots, \alpha_k \in \mathbb N$,
\[ \lim_{n \to \infty} \frac{\{1 \le i \le n \, : \, a_i = \alpha_1, \ldots, a_{i+k-1} = \alpha_k \} }{n} = \prod_{j=1}^k L_{\alpha_j}.\]

\medskip
Normal sequences in this setting have already been considered in the case of {\em L\"uroth expansions}, which were introduced by L\"uroth in 1883 \cite{Lur83}. L\"uroth expansions of  numbers $x \in (0,1]$ are expressions of the form
\begin{equation}\label{q:lurothexp}
x = \sum_{n \ge 1} \frac{a_n}{ \prod_{i=1}^n a_i(a_i+1)}, \quad a_n \in \mathbb N,
\end{equation}
where the sum can either be finite or infinite. L\"uroth expansions are generated by iterating the transformation
\[ T_L: (0,1]\to (0,1], \, x \mapsto \left\lfloor \frac1x \right\rfloor \left( \left\lfloor \frac1x \right\rfloor +1 \right) x - \left\lfloor \frac1x \right\rfloor,\]
see Figure~\ref{f:luroth} for the graph. Setting for each $n \ge 1$ that $a_n(x) = a_1(T_L^{n-1}(x)) = \left\lfloor \frac1{T_L^{n-1}(x)} \right\rfloor $, one obtains the expression from \eqref{q:lurothexp}. As $a_n (x) = k$ if $T^{n-1}_L(x) \in (\frac1{k+1}, \frac1k \big]$ we expect the digit $k \in \mathbb N$ to appear with frequency $\frac1{k(k+1)}$. A L\"uroth normal number would therefore be a number $x \in (0,1]$ for which a sequence $(a_i)_{i \in \mathbb N} \in \mathbb N^\mathbb N$ exists such that \eqref{q:lurothexp} holds and $(a_i)_{i \in \mathbb N}$ is $L$-normal with respect to the probability sequence $L=\big(\frac1{k(k+1)}\big)_{k \in \mathbb N}$. Compared to base $N$-expansions, the two additional challenges in finding $L$-normal sequences are that the infinitely many different digits occur with different frequencies and that one cannot set up a process that adds blocks of digits according to their length, i.e., a straightforward generalisation of Champernowne's construction is impossible.

\medskip
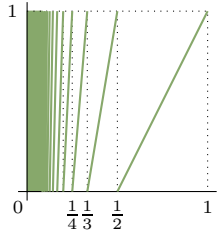
\begin{figure}[h]
\begin{tikzpicture}[scale=2.4]
\filldraw[fill=asparagus, draw=asparagus] (0,0) rectangle (1/12,1);
\draw(-.05,0)node[below]{\scriptsize 0}--(1/4,0)node[below]{\small $\frac14$}--(1/3,0)node[below]{\small $\frac13$}--(1/2,0)node[below]{\small $\frac12$}--(1,0)node[below]{\scriptsize 1}--(1.05,0)(0,-.05)--(0,1)node[left]{\scriptsize 1}--(0,1.05);
\draw[dotted](1/2,0)--(1/2,1)(1/3,0)--(1/3,1)(1/4,0)--(1/4,1)(1/5,0)--(1/5,1)(0,1)--(1,1)--(1,0);
\draw[thick, asparagus](.5,0)--(1,1)(1/3,0)--(1/2,1)(1/4,0)--(1/3,1)(1/5,0)--(1/4,1)(1/6,0)--(1/5,1)(1/7,0)--(1/6,1)(1/8,0)--(1/7,1)(1/9,0)--(1/8,1)(1/10,0)--(1/9,1)(1/11,0)--(1/10,1)(1/12,0)--(1/11,1);
\end{tikzpicture}
\caption{The graph of the L\"uroth transformation $T_L$.}
\label{f:luroth}
\end{figure}

\medskip
Normal numbers for L\"uroth expansions and the more general setting of so-called GLS-expansions are considered in \cite{Bok09,AP15,MM16,DLR20}. In \cite{Bok09,MM16} $L$-normal numbers for the probability sequence $L= \big(\frac1{k(k+1)}\big)_{k \in \mathbb N}$ and with the L\"uroth projection
\[ (a_i)_{i \in \mathbb N} \mapsto  \sum_{n \ge 1} \frac{a_n}{ \prod_{i=1}^n a_i(a_i+1)}\]
are considered. In the bachelor thesis \cite{Bok09} (supervised by Kraaikamp) Boks described a possible construction for an $L$-normal sequence, but the thesis does not contain a proof of normality. In \cite{MM16} Madritsch and Mance gave a quite general construction for $L$-normal sequences. It resembles Champernowne's construction in that one first adds many copies of all possible blocks of length 2 with digits 2 and 3 (so 22, 23, 32 and 33), then many copies of all blocks of length 3 with digits 2, 3 and 4, etcetera. L\"uroth expansions and base $N$-expansions are both special cases of {\em Generalised L\"uroth Series (GLS) expansions}, introduced in \cite{BBDK96}. These GLS expansions correspond to different probability sequences $L$ on the digit set $\mathbb N$ in case of infinitely many digits or to probability vectors $(L_d)_{1 \le d \le N}$ in the case of GLS expansions with $N$ digits, $N \in \mathbb N$. In \cite{Van14} Vandehey considered normal numbers for GLS expansions with finite digit sets.  In \cite{AP15} Aehle and Paulsen addressed the question of finding numbers with normal GLS expansions for infinite digit sets and established a connection with equidistributed sequences. In \cite{DLR20} Dajani, De Lepper and Robinson found an $L$-normal sequence for the probability sequence $L= \left(\frac1{2^n} \right)_{n \in \mathbb N }$.

\medskip
Our main result is that we provide a family of $L$-normal sequences for any ordered probability sequence $L$. The construction for this family is based on the binary tree introduced in \cite{Bok09} for the sequence $L= \big(\frac1{k(k+1)}\big)_{k \in \mathbb N}$, which we have adapted to also work for other sequences $L$, in particular those where the numbers $L_d$ are not necessarily rational. As in \cite{MM16} we also add many copies of blocks of digits, but we do not add them by increasing block length and compared to \cite{MM16} our construction requires adding significantly fewer copies of the blocks. After having constructed $L$-normal sequences we project these sequences to real numbers that have a normal GLS expansions, of which normal L\"uroth expansions are a particular case.

\medskip
As for L\"uroth expansions, many types of number expansions are generated by dynamical systems which have an invariant measure that is ergodic and absolutely continuous with respect to Lebesgue measures. (In case of the L\"uroth transformation, Lebesgue measure itself is invariant and ergodic.) The Birkhoff ergodic theorem  then implies that the proportion of time that a typical orbit spends in any measurable subset of the interval $[0,1]$ equals the measure of that subset. Normal numbers are explicit examples of such typical orbits and one can therefore also consider them for certain dynamical systems without a strong association to number expansions. In the last section we illustrate this with normal numbers for higher dimensional GLS transformations. 

\medskip
The article is outlined as follows. In Section~\ref{s:tree} we give some preliminaries, we introduce the rooted binary tree on which our construction is based and we prove some properties of the tree. In Section~\ref{s:normal} we provide the construction of a family of $L$-normal sequences and give the proof of normality. Section~\ref{s:GLS} is dedicated to relating the $L$-normal sequences found in Section~\ref{s:normal} to normal numbers with respect to GLS expansions in one and higher dimensions.

\section{The $L$-tree}\label{s:tree}
\subsection{Notation for words and sequences}
Throughout the text we use the notation $\llbracket n,k \rrbracket$ for intervals in the integers in the following way: for $n,k \in \mathbb Z$ let
\[ \llbracket n,k \rrbracket = \{ i \in \mathbb Z \, : \, n \le i \le k\}.\]
We consider words and sequences in the alphabet $\mathcal A= \mathbb N$. For clarity of notation, we use $\mathcal A$ instead of $\mathbb N$ when we refer to $\mathbb N$ as the digit set. For each $n \ge 1$ the set $\mathcal A^n$ contains all {\em words} of length $n$, which are all strings $\bm{u} = u_1 \ldots u_n$ with $u_i \in \mathcal A$ for all $1 \le i \le n$. The set $\mathcal A^0$ is the set containing only the empty word, which we denote by $\epsilon$. For us it will be convenient to let $\mathcal A^* = \bigcup_{n \ge 1} \mathcal A^n$ be the set of all words with digits in $\mathcal A$ and let $\mathcal A^*_\epsilon =  \bigcup_{n \ge 0} \mathcal A^n$ be the set of all words with digits in $\mathcal A$ together with the empty word. Let $|\bm{u}|$ denote the length of $\bm{u}$, so $|\bm{u}|=n$ if $\bm{u} \in \mathcal A^n$, $n \ge 0$. If $\bm u = u_1 \ldots u_m \in \mathcal A^m$ for some $m \ge 1$ and $1 \le n \le m$, then we let $\bm{u}|_n = u_1\cdots u_n$. For a word $\alpha \in \mathcal A^*$ we let
\begin{equation}\label{q:cylinder}
\mathcal A(\alpha) = \{ \bm{u} \in \mathcal A^* \, : \, \bm{u}|_{|\alpha|} = \alpha\}
\end{equation}
be the set of words that start with $\alpha$. For two words $\bm{u}, \bm{v} \in \mathcal A^*$ we use $\bm{uv}$ to denote the concatenation. For a word $\bm{u} = u_1\ldots u_n \in \mathcal A^n$, $n \ge 1$, set $\bm{u}^+ = u_1 \ldots u_{n-1}(u_n+1)$. For two words $\bm{u} = u_1\cdots u_n, \bm{v} = v_1\cdots v_k \in \mathcal A^*$ we say that $\bm{v}$ is a {\em subword} of $\bm{u}$ if there is a $1 \le j \le n$ such that $u_j \cdots u_{j+k-1} = \bm{v}$ and $\bm{v}$ is a {\em prefix} of $\bm{u}$ if it is a subword with $j=1$. In particular this implies that $1 \le k \le n$. %In case $j=1$ we call $\bm{v}$ a {\em prefix} of $\bm{u}$.

\medskip
%A {\em cylinder set} corresponding to a word $\bm{u} \in \mathcal A^*$ is the set of sequences
%\[ C(\bm{u}) = \{ \bm{a} \in \mathcal A^\mathbb N \, : \, a_i = u_i, \, 1 \le i \le |\bm{u}| \}.\]
Let $\mathcal A^\mathbb N$ be the set of all one-sided infinite sequences $ (a_i)_{i \ge 1}$ with $a_i \in \mathcal A$ for all $i \ge 1$. Throughout the text we fix a sequence $L = (L_d)_{d \in \mathbb N}$ with $0 < L_{d+1} \le L_d < 1$ for all $d \in \mathbb N$ and $\sum_{d \in \mathbb N}L_d =1$. The numbers $L_d$ represent the probabilities with which we would like to see each of the digits $d \in \mathcal A$ appear in a sequence $A \in \mathcal A^\mathbb N$. Note that the ordering of the entries of $L$ according to size is just for notational convenience, since by relabelling we could obtain any redistribution of the probabilities over the digits. For a word $\alpha = \alpha_1 \cdots \alpha_k \in \mathcal A^*$ we set
\[ \mu(\alpha) := \prod_{i=1}^k L_{\alpha_i}\]
and we let $\mu(\epsilon)=1$. For any $\alpha \in \mathcal A^*$ and subset $T \subseteq \mathbb N$ define
\[ N_A(\alpha, T):= \frac{\#\{i \in T\, :\,  a_i=\alpha_1,  \ldots,   a_{i+k-1}= \alpha_k\}}{\#T}.\]
We call a sequence $A \in \mathcal A^\mathbb N$ {\em $L$-normal} if for all $\alpha \in \mathcal A^*$,
\[ \lim_{n \to \infty} N_A(\alpha,\llbracket 1,n\rrbracket)=\mu(\alpha).\]

\subsection{Construction of the tree}
For our fixed probability sequence $L = (L_d)_{d \in \mathbb N}$ we construct a rooted, labeled, binary probability tree, which is strongly reminiscent of the tree introduced in \cite{Bok09} and which we call the $L$-tree. For each $n \ge 1$ set
\[ p_n := \frac{L_n}{1-\sum_{i=1}^{n-1}L_i} \quad \text{ and } \quad q_n := 1-p_n = \frac{1-\sum_{i=1}^nL_i}{1-\sum_{i=1}^{n-1}L_i},\]
where we let $\sum_{i=1}^0 L_i =0$. One can think of %follows. For $n=1$ we have $p_1=L_1$ and $q_1 = 1-L_1$, so $p_1$ is the probability of seeing the digit 1 and $q_1$ is the probability of seeing a digit that is not 1. Then $p_2 = \frac{L_2}{1-L_1}$ is the probability of seeing the digit 2 while knowing that the digit is not 1 and $q_2 = \frac{1-L_1-L_2}{1-L_1}$ is the probability of not seeing a 2 while knowing that the digit is not a 1. In general,
$p_n$ as the probability of seeing the digit $n$ while knowing that the digit is at least $n$ and of $q_n$ as the probability of seeing a digit unequal to $n$ while knowing that the digit is at least $n$.

\medskip
The set of nodes of the $L$-tree is $\mathcal A^*$ with root 1. The set of edges $\mathcal E$ is given by
\[ \mathcal E = \{ (\bm{u}, \bm{u}1), (\bm{u}, \bm{u}^+) \, : \, \bm{u} \in \mathcal A^* \}.\]
To be more precise, each node $\bm{u}=u_1 \ldots u_n$ has two children: a left child with label $\bm{u}1$ and a right child with label $\bm{u}^+$.  We consider the edge labelling of the tree given by the map
\[ f: \mathcal E  \to [0,1] ; \, (\bm{u},\bm{v}) \mapsto \begin{cases}
p_{u_n}, & \text{if } \bm{u} = u_1\ldots u_n, \, \bm{v} = \bm{u}1,\\
q_{u_n}, & \text{if } \bm{u} = u_1\ldots u_n, \, \bm{v} = \bm{u}^+.
\end{cases}\]
Thus for $n \ge 1$ the $L$-tree uses the rule
\begin{center}
\begin{forest}
for tree={%
    s sep'+=5mm,
    l'+=5mm,
    }
[ $ \text{$u_1u_2\dots  u_n$} $
   [ {$ \text{$u_1u_2\dots  u_n1 $} $},  tier=een, edge label={node[midway, above=1mm, left]{$p_{u_n}$}}
   ]
   [ {$\text{$u_1u_2\dots  u_n^+.$} $},  tier=een,  edge label={node[midway, above=1mm,  right]{$q_{u_n}$}}
   ]
]
\end{forest}
\end{center}
Figure~\ref{f:tree} shows the first six levels of the $L$-tree with the nodes written vertically. Note that for each $\bm{u} \in \mathcal A^*_\epsilon$ and each $\bm{v} \in \mathcal A^*$ it holds that
\begin{equation}\label{q:invariant}
f(\bm{uv}, \bm{uv}1) = f(\bm{v}, \bm{v}1) \quad \text{ and } \quad f(\bm{uv}, \bm{uv}^+) = f(\bm{v}, \bm{v}^+).
\end{equation}

\begin{center}
\begin{figure}[h]
\begin{tikzpicture}[scale=.95]
\node at (0,0) {\scriptsize 1};
\draw (-.3,-.1)--(-3.7,-.9)(.3,-.1)--(3.7,-.9);
\node at (-1.8,-.2) {\color{bittersweet}\scriptsize $p_1$};
\node at (1.8,-.2) {\color{bittersweet}\scriptsize $q_1$};
\node at (-4,-1) {\scriptsize \begin{tabular}{c}1 \\ 1\end{tabular}};
\node at (4,-1) {\scriptsize 2};
\draw (-4.3,-1.1)--(-5.7,-1.8)(-3.7,-1.1)--(-2.3,-1.8);
\node at (-5.1,-1.25) {\color{bittersweet}\scriptsize $p_1$};
\node at (-2.9,-1.25) {\color{bittersweet}\scriptsize $q_1$};
\node at (-6,-2) {\scriptsize \begin{tabular}{c}1 \\ 1\\1 \end{tabular}};
\node at (-2,-2) {\scriptsize \begin{tabular}{c}1 \\ 2\end{tabular}};
\draw (4.2,-1.2)--(5.8,-1.8)(3.8,-1.2)--(2.2,-1.8);
\node at (6,-2) {\scriptsize 3};
\node at (2,-2) {\scriptsize \begin{tabular}{c}2 \\ 1\end{tabular}};
\node at (5.1,-1.25) {\color{bittersweet}\scriptsize $q_2$};
\node at (2.9,-1.25) {\color{bittersweet}\scriptsize $p_2$};
\draw(-6.2,-2.2)--(-6.7,-2.7)(-5.8,-2.2)--(-5.3,-2.7);
\node at (-6.55,-2.25) {\color{bittersweet}\scriptsize $p_1$};
\node at (-5.4,-2.25) {\color{bittersweet}\scriptsize $q_1$};
\node at (-7,-3) {\scriptsize \begin{tabular}{c}1 \\ 1\\ 1\\ 1\end{tabular}};
\node at (-5,-3) {\scriptsize \begin{tabular}{c}1 \\ 1\\ 2\end{tabular}};
\draw(-2.2,-2.2)--(-2.7,-2.7)(-1.8,-2.2)--(-1.3,-2.7);
\node at (-2.6,-2.25) {\color{bittersweet}\scriptsize $p_2$};
\node at (-1.4,-2.25) {\color{bittersweet}\scriptsize $q_2$};
\node at (-3,-3) {\scriptsize \begin{tabular}{c}1\\ 2\\ 1\end{tabular}};
\node at (-1,-3) {\scriptsize \begin{tabular}{c}1\\ 3\end{tabular}};
\draw(6.2,-2.2)--(6.7,-2.7)(5.8,-2.2)--(5.3,-2.7);
\node at (7,-3) {\scriptsize 4};
\node at (5,-3) {\scriptsize \begin{tabular}{c}3 \\ 1\end{tabular}};
\node at (6.6,-2.25) {\color{bittersweet}\scriptsize $q_3$};
\node at (5.4,-2.25) {\color{bittersweet}\scriptsize $p_3$};
\draw(2.2,-2.2)--(2.7,-2.7)(1.8,-2.2)--(1.3,-2.7);
\node at (3,-3) {\scriptsize \begin{tabular}{c}2 \\ 2\end{tabular}};
\node at (1,-3) {\scriptsize \begin{tabular}{c}2\\ 1 \\ 1\end{tabular}};
\node at (2.6,-2.25) {\color{bittersweet}\scriptsize $q_1$};
\node at (1.4,-2.25) {\color{bittersweet}\scriptsize $p_1$};

\draw(-7.2,-3.2)--(-7.4,-3.7)(-6.8,-3.2)--(-6.6,-3.7);
\node at (-7.5,-4.5) {\scriptsize \begin{tabular}{c}1 \\ 1\\ 1\\ 1\\ 1\end{tabular}};
\node at (-6.5,-4.35) {\scriptsize \begin{tabular}{c}1 \\ 1\\ 1\\ 2\end{tabular}};
\node at (-7.5,-3.35) {\color{bittersweet}\scriptsize $p_1$};
\node at (-6.5,-3.35) {\color{bittersweet}\scriptsize $q_1$};
\draw(-5.2,-3.2)--(-5.4,-3.7)(-4.8,-3.2)--(-4.6,-3.7);
\node at (-5.5,-4.35) {\scriptsize \begin{tabular}{c}1 \\ 1\\ 2 \\ 1\end{tabular}};
\node at (-4.5,-4.2) {\scriptsize \begin{tabular}{c}1\\ 1\\ 3\end{tabular}};
\node at (-5.5,-3.35) {\color{bittersweet}\scriptsize $p_2$};
\node at (-4.5,-3.35) {\color{bittersweet}\scriptsize $q_2$};

\draw(-3.2,-3.2)--(-3.4,-3.7)(-2.8,-3.2)--(-2.6,-3.7);
\node at (-3.5,-4.45) {\scriptsize \begin{tabular}{c}1\\ 2\\ 1\\ 1\end{tabular}};
\node at (-2.5,-4.3) {\scriptsize \begin{tabular}{c}1\\ 2\\ 2\end{tabular}};
\node at (-3.5,-3.35) {\color{bittersweet}\scriptsize $p_1$};
\node at (-2.5,-3.35) {\color{bittersweet}\scriptsize $q_1$};
\draw(-1.2,-3.2)--(-1.4,-3.7)(-.8,-3.2)--(-.6,-3.7);
\node at (-1.5,-4.3) {\scriptsize \begin{tabular}{c}1\\ 3\\ 1\end{tabular}};
\node at (-.5,-4.15) {\scriptsize \begin{tabular}{c}1\\ 4\end{tabular}};
\node at (-1.5,-3.35) {\color{bittersweet}\scriptsize $p_3$};
\node at (-.5,-3.35) {\color{bittersweet}\scriptsize $q_3$};

\draw(7.2,-3.2)--(7.4,-3.7)(6.8,-3.2)--(6.6,-3.7);
\node at (7.5,-4) {\scriptsize \begin{tabular}{c}5\end{tabular}};
\node at (6.5,-4.15) {\scriptsize \begin{tabular}{c}4 \\ 1\end{tabular}};
\node at (7.5,-3.35) {\color{bittersweet}\scriptsize $q_4$};
\node at (6.5,-3.35) {\color{bittersweet}\scriptsize $p_4$};
\draw(5.2,-3.2)--(5.4,-3.7)(4.8,-3.2)--(4.6,-3.7);
\node at (5.5,-4.15) {\scriptsize \begin{tabular}{c}3 \\ 2 \end{tabular}};
\node at (4.5,-4.3) {\scriptsize \begin{tabular}{c}3\\ 1\\ 1 \end{tabular}};
\node at (5.5,-3.35) {\color{bittersweet}\scriptsize $q_1$};
\node at (4.5,-3.35) {\color{bittersweet}\scriptsize $p_1$};

\draw(3.2,-3.2)--(3.4,-3.7)(2.8,-3.2)--(2.6,-3.7);
\node at (3.5,-4.15) {\scriptsize \begin{tabular}{c}2\\ 3\end{tabular}};
\node at (2.5,-4.3) {\scriptsize \begin{tabular}{c}2\\ 2\\ 1\end{tabular}};
\node at (3.5,-3.35) {\color{bittersweet}\scriptsize $q_2$};
\node at (2.5,-3.35) {\color{bittersweet}\scriptsize $p_2$};
\draw(1.2,-3.2)--(1.4,-3.7)(.8,-3.2)--(.6,-3.7);
\node at (1.5,-4.3) {\scriptsize \begin{tabular}{c}2\\ 1\\ 2\end{tabular}};
\node at (.5,-4.45) {\scriptsize \begin{tabular}{c}2\\ 1\\ 1\\ 1\end{tabular}};
\node at (.5,-3.35) {\color{bittersweet}\scriptsize $p_1$};
\node at (1.5,-3.35) {\color{bittersweet}\scriptsize $q_1$};

\draw(-7.7,-5)--(-7.8,-5.5)(-7.3,-5)--(-7.2,-5.5);
\draw(-6.7,-5)--(-6.8,-5.5)(-6.3,-5)--(-6.2,-5.5);
\node at (-7.8,-6.5) {\scriptsize \begin{tabular}{c}1 \\ 1\\ 1\\ 1\\ 1\\ 1\end{tabular}};
\node at (-7.2,-6.35) {\scriptsize \begin{tabular}{c}1 \\ 1\\ 1\\ 1\\ 2\end{tabular}};
\node at (-6.8,-6.35) {\scriptsize \begin{tabular}{c}1 \\ 1\\ 1\\ 2\\ 1 \end{tabular}};
\node at (-6.2,-6.2) {\scriptsize \begin{tabular}{c}1 \\ 1\\ 1\\ 3\end{tabular}};
\node at (-7.9,-5.1) {\color{bittersweet}\scriptsize $p_1$};
\node at (-7.1,-5.1) {\color{bittersweet}\scriptsize $q_1$};
\node at (-6.97,-5.4) {\color{bittersweet}\scriptsize $p_2$};
\node at (-6.03,-5.4) {\color{bittersweet}\scriptsize $q_2$};

\draw(-5.7,-5)--(-5.8,-5.5)(-5.3,-5)--(-5.2,-5.5);
\draw(-4.6,-4.8)--(-4.8,-5.5)(-4.4,-4.8)--(-4.2,-5.5);
\node at (-5.8,-6.35) {\scriptsize \begin{tabular}{c}1 \\ 1\\ 2 \\ 1\\ 1 \end{tabular}};
\node at (-5.2,-6.2) {\scriptsize \begin{tabular}{c}1\\ 1\\ 2\\ 2 \end{tabular}};
\node at (-4.8,-6.2) {\scriptsize \begin{tabular}{c}1 \\ 1\\ 3 \\ 1 \end{tabular}};
\node at (-4.2,-6.05) {\scriptsize \begin{tabular}{c}1\\ 1\\ 4 \end{tabular}};
\node at (-5.9,-5.1) {\color{bittersweet}\scriptsize $p_1$};
\node at (-5.1,-5.1) {\color{bittersweet}\scriptsize $q_1$};
\node at (-4.97,-5.4) {\color{bittersweet}\scriptsize $p_3$};
\node at (-4.03,-5.4) {\color{bittersweet}\scriptsize $q_3$};

\draw(-3.7,-5)--(-3.8,-5.5)(-3.3,-5)--(-3.2,-5.5);
\draw(-2.6,-4.8)--(-2.8,-5.5)(-2.4,-4.8)--(-2.2,-5.5);
\node at (-3.8,-6.35) {\scriptsize \begin{tabular}{c}1 \\ 2\\ 1\\ 1\\ 1\end{tabular}};
\node at (-3.2,-6.2) {\scriptsize \begin{tabular}{c}1 \\ 2\\ 1\\ 2\end{tabular}};
\node at (-2.8,-6.2) {\scriptsize \begin{tabular}{c}1 \\ 2\\ 2\\ 1 \end{tabular}};
\node at (-2.2,-6.05) {\scriptsize \begin{tabular}{c}1 \\ 2 \\ 3\end{tabular}};
\node at (-3.9,-5.1) {\color{bittersweet}\scriptsize $p_1$};
\node at (-3.1,-5.1) {\color{bittersweet}\scriptsize $q_1$};
\node at (-2.97,-5.4) {\color{bittersweet}\scriptsize $p_2$};
\node at (-2.03,-5.4) {\color{bittersweet}\scriptsize $q_2$};

\draw(-1.6,-4.8)--(-1.8,-5.5)(-1.4,-4.8)--(-1.2,-5.5);
\draw(-.6,-4.8)--(-.8,-5.5)(-.4,-4.8)--(-.2,-5.5);
\node at (-1.8,-6.2) {\scriptsize \begin{tabular}{c}1 \\ 3 \\ 1\\ 1 \end{tabular}};
\node at (-1.2,-6.05) {\scriptsize \begin{tabular}{c}1\\ 3\\ 2 \end{tabular}};
\node at (-.8,-6.05) {\scriptsize \begin{tabular}{c}1 \\ 4 \\ 1 \end{tabular}};
\node at (-.2,-5.9) {\scriptsize \begin{tabular}{c}1\\ 5 \end{tabular}};
\node at (-1.85,-5.05) {\color{bittersweet}\scriptsize $p_1$};
\node at (-1.15,-5.05) {\color{bittersweet}\scriptsize $q_1$};
\node at (-.97,-5.4) {\color{bittersweet}\scriptsize $p_4$};
\node at (-.03,-5.4) {\color{bittersweet}\scriptsize $q_4$};

\draw(7.6,-4.2)--(7.8,-5.5)(7.4,-4.2)--(7.2,-5.5);
\draw(6.6,-4.5)--(6.8,-5.5)(6.4,-4.5)--(6.2,-5.5);
\node at (7.8,-5.75) {\scriptsize \begin{tabular}{c}6\end{tabular}};
\node at (7.2,-5.9) {\scriptsize \begin{tabular}{c}5 \\ 1\end{tabular}};
\node at (6.8,-5.9) {\scriptsize \begin{tabular}{c}4 \\ 2 \end{tabular}};
\node at (6.2,-6.05) {\scriptsize \begin{tabular}{c}4 \\ 1\\1 \end{tabular}};
\node at (7.97,-5.4) {\color{bittersweet}\scriptsize $q_5$};
\node at (7.03,-5.4) {\color{bittersweet}\scriptsize $p_5$};
\node at (6.87,-4.95) {\color{bittersweet}\scriptsize $q_1$};
\node at (6.15,-4.95) {\color{bittersweet}\scriptsize $p_1$};

\draw(5.6,-4.5)--(5.8,-5.5)(5.4,-4.5)--(5.2,-5.5);
\draw(4.6,-4.8)--(4.8,-5.5)(4.4,-4.8)--(4.2,-5.5);
\node at (5.8,-5.9) {\scriptsize \begin{tabular}{c}3 \\ 3 \end{tabular}};
\node at (5.2,-6.05) {\scriptsize \begin{tabular}{c}3\\ 2\\ 1 \end{tabular}};
\node at (4.8,-6.05) {\scriptsize \begin{tabular}{c}3 \\  1\\ 2  \end{tabular}};
\node at (4.2,-6.2) {\scriptsize \begin{tabular}{c}3\\ 1\\ 1\\1 \end{tabular}};
\node at (5.97,-5.4) {\color{bittersweet}\scriptsize $q_2$};
\node at (5.03,-5.4) {\color{bittersweet}\scriptsize $p_2$};
\node at (4.85,-5.05) {\color{bittersweet}\scriptsize $q_1$};
\node at (4.15,-5.05) {\color{bittersweet}\scriptsize $p_1$};

\draw(3.6,-4.8)--(3.8,-5.5)(3.4,-4.8)--(3.2,-5.5);
\draw(2.6,-4.8)--(2.8,-5.5)(2.4,-4.8)--(2.2,-5.5);
\node at (3.8,-5.9) {\scriptsize \begin{tabular}{c}2 \\ 4\end{tabular}};
\node at (3.2,-6.05) {\scriptsize \begin{tabular}{c}2 \\ 3\\ 1\end{tabular}};
\node at (2.8,-6.05) {\scriptsize \begin{tabular}{c} 2\\ 2\\ 2 \end{tabular}};
\node at (2.2,-6.2) {\scriptsize \begin{tabular}{c}2 \\ 2 \\ 1\\ 1\end{tabular}};
\node at (3.97,-5.4) {\color{bittersweet}\scriptsize $q_3$};
\node at (3.03,-5.4) {\color{bittersweet}\scriptsize $p_3$};
\node at (2.85,-5.05) {\color{bittersweet}\scriptsize $q_1$};
\node at (2.15,-5.05) {\color{bittersweet}\scriptsize $p_1$};

\draw(1.6,-4.8)--(1.8,-5.5)(1.4,-4.8)--(1.2,-5.5);
\draw(.7,-5)--(.8,-5.5)(.3,-5)--(.2,-5.5);
\node at (1.8,-6.05) {\scriptsize \begin{tabular}{c}2 \\ 1 \\ 3 \end{tabular}};
\node at (1.2,-6.2) {\scriptsize \begin{tabular}{c}2\\ 1\\ 2\\ 1 \end{tabular}};
\node at (.8,-6.2) {\scriptsize \begin{tabular}{c}2 \\ 1 \\ 1\\ 2 \end{tabular}};
\node at (.2,-6.35) {\scriptsize \begin{tabular}{c}2\\ 1\\ 1\\ 1\\ 1 \end{tabular}};
\node at (1.97,-5.4) {\color{bittersweet}\scriptsize $q_2$};
\node at (1.03,-5.4) {\color{bittersweet}\scriptsize $p_2$};
\node at (.9,-5.15) {\color{bittersweet}\scriptsize $q_1$};
\node at (.1,-5.15) {\color{bittersweet}\scriptsize $p_1$};

\node at (0,-7.5) {$\vdots$};
\end{tikzpicture}
\caption{The first six levels of the  $L$-tree.}
\label{f:tree}
\end{figure}
\end{center}

For each word $\bm{u}=u_1 \ldots u_n \in \mathcal A^*_\epsilon$, define the {\em depth} of $\bm{u}$ by
\[ D(\bm{u}) = \sum_{i=1}^n u_i,\]
with $\sum_{i=1}^0 u_i=0$. Then $D(\bm{u})$ is the number of nodes in the path starting from the root 1 and ending in the node $\bm{u}$. Immediately from the definition we obtain the identity
\begin{equation}\label{q:Ds}
D(\bm{uv}) = D(\bm{u}) + D(\bm{v}), \quad  \bm{u},\bm{v} \in \mathcal A^*_\epsilon.
\end{equation}
For what follows later it is useful to introduce the notation $\bm{u}^{-2}$ for the {\em grandparent} of a word $\bm{u} \in \mathcal A^*$ with $D(\bm{u}) \ge 3$, which is the unique word $\bm{v} \in \mathcal A^*$ such that $\bm{u} \in \{ \bm{v}11, \bm{v}2, \bm{v}^+1, (\bm{v}^+)^+\}$. For each $n \ge 1$ write $G(n)$ for the collection of words that appear at depth $n$ of the tree, so
\[ G(n) = \{ \bm{u} \in \mathcal A^* \, : \, D(\bm{u})=n\}.\]
For any word $\bm{u}=u_1\ldots u_n \in \mathcal A^*$, write $1=\bm{u}_1, \bm{u}_2, \ldots, \bm{u}_{D(\bm{u})}=\bm{u}$ for the sequence of nodes in the path from 1 to $\bm{u}$. Then
\[ W(\bm{u}) = \prod_{i=1}^{D(\bm{u})-1} f(\bm{u}_i, \bm{u}_{i+1})\]
denotes the product of the probabilities along the path from 1 to $\bm{u}$, which we call the {\em weight} of $\bm{u}$. Here we let $\prod_{i=1}^{0} f(\bm{u}_i, \bm{u}_{i+1})=1$. We collect the some properties of the weight $W$.

\begin{lemma}\label{l:weight}
The following statements hold.
\begin{itemize}
\item[(i)] For any $n \in \mathbb N$, $\sum_{\bm{u} \in G(n)} W(\bm{u})=1$.
\vspace{.15cm}
\item[(ii)] For any $\bm{u},\bm{v} \in \mathcal A_{\epsilon}^*$, $W(\bm{uv}) = W(\bm{u}1)W(\bm{v})$.
\vspace{.15cm}
\item[(iii)] For any $\bm{u} \in \mathcal A_{\epsilon}^*$ $W(\bm{u}1) = \mu(\bm{u})$.
\end{itemize}
\end{lemma}

\begin{proof}
For (i), fix an $n \in \mathbb N$. Since $p_k + q_k=1$ for each $k \in \mathbb N$, it holds for each $\bm{u} \in G(k)$ that
\[ W(\bm{u}) = W(\bm{u}1) + W(\bm{u}^+).\]
Therefore,
\begin{equation}\label{q:one}\begin{split}
\sum_{\bm{u} \in G(n)} W(\bm{u}) =\ & \sum_{\bm{u} \in G(n-1)} W(\bm{u}1) + \sum_{\bm{u} \in G(n-1)} W(\bm{u}^+)\\
=\ & \sum_{\bm{u} \in G(n-1)} W(\bm{u}) = \cdots = \sum_{\bm{u} \in G(1)} W(\bm{u})  =1.
\end{split}
\end{equation}
For (ii), using \eqref{q:invariant} we obtain
\[ \begin{split}
W(\bm{uv})=\ & \prod_{j=1}^{D(\bm{uv})-1} f\big((\bm{uv})_j, (\bm{uv})_{j+1}\big)\\
=\ & \left( \prod_{j=1}^{D(\bm{u})-1} f(\bm{u}_j, \bm{u}_{j+1}) \right) \cdot f(\bm{u}, \bm{u}1) \cdot \left( \prod_{j=D(\bm{u})+1}^{D(\bm{uv})-1} f\big((\bm{uv})_j, (\bm{uv})_{j+1}\big)\right)\\
=\ & \prod_{j=1}^{D(\bm{u}1)-1} f\big((\bm{u}1)_j, (\bm{u}1)_{j+1}\big) \prod_{j=1}^{D(\bm{v})-1} f(\bm{v}_j, \bm{v}_{j+1})\\
=\ & W(\bm{u}1)W(\bm{v}).
\end{split}\]
For (iii) we first consider a single digit $d \in \mathcal A$ and note that
\begin{equation}\label{extra}
W(d) = \prod_{j=1}^{d-1} q_j = 1-\sum_{j=1}^{d-1}L_j.
\end{equation}
This gives $W(d1) = p_d \prod_{j=1}^{d-1} q_j = L_d$. Then for $\bm{u} = u_1 \ldots u_n$ by repeated application of this fact and (ii) we get 
\[ \begin{split}
W(\bm{u}1) =\ & W(u_1 \ldots u_{n-1} 1)W(u_n1) = W(u_1 \ldots u_{n-2} 1)W(u_{n-1}1)L_{u_n} \\
 =\ & \cdots = \prod_{j=1}^n W(u_j1) =\prod_{j=1}^n L_{u_j} = \mu(\bm{u}). \qedhere
 \end{split}\]
\end{proof}

The following lemma will be useful later.
\begin{lemma}\label{l:unbounded} The following holds:
\[ \lim_{n \to \infty} \sum_{\bm{u} \in G(n)} W(\bm{u}) |\bm{u}| = \infty.\]
\end{lemma}

\begin{proof}
For ease of notation set $\Delta_n = \sum_{\bm{u} \in G(n)} W(\bm{u}) |\bm{u}| $, $n \in \mathbb N$. Fix an $n \in \mathbb N$ and observe that for each $d\in \llbracket 1, n-1 \rrbracket$ it holds that
\[\{\bm{v} \in \mathcal A^*\, :\,  d \bm{v} \in G(n) \}= G(n-d).\]
By isolating the first digit in the words from $G(n)$ and using Lemma~\ref{l:weight} and \eqref{extra} we have
\begin{equation*}
	\begin{split}
		\Delta_n =\ & W(n) + \sum_{d=1}^{n-1} \sum_{\bm{u} \in G(n): \bm{u}=d\bm{v}} W(\bm{u}) |\bm{u}| \\
  =\ & W(n) + \sum_{d=1}^{n-1}  \sum_{\bm{v} \in G(n-d)}W(d1)W(\bm{v}) (|\bm{v}|+1) \\
					   =\ & W(n) + \sum_{d=1}^{n-1} L_d \sum_{\bm{v} \in G(n-d)}W(\bm{v}) (|\bm{v}|+1) \\
					   =\ & W(n) + \sum_{d=1}^{n-1} L_d \left( 1+ \sum_{\bm{v} \in G(n-d)}W(\bm{v}) |\bm{v}| \right) \\
					   =\ & W(n) + \sum_{d=1}^{n-1} L_d ( 1+ \Delta_{n-d}) \\
					  =\ & 1-\sum_{d=1}^{n-1} L_d +  \sum_{d=1}^{n-1} L_d  + \sum_{d=1}^{n-1} L_d \Delta_{n-d}  \\
					  =\ & 1 + \sum_{d=1}^{n-1} L_d \Delta_{n-d}.
	\end{split}
\end{equation*}
As $\Delta_1=1,$ we have $\Delta_2 > \Delta_1$.  Now let $N \in \mathbb N_{\geq 2},$ and suppose that $\Delta_{n} > \Delta_{n-1}$ for all $n \in \llbracket 2,N \rrbracket$. Then
\[ \Delta_{N+1} = 1+ \sum_{d=1}^N L_d \Delta_{N+1-d}  > 1 + \sum_{d=1}^{N-1} L_d \Delta_{N-d} + L_N\Delta_1  = \Delta_N + L_N \Delta_1 > \Delta_N.\]
Hence, the sequence $(\Delta_n)_{n \in \mathbb N}$ is strictly increasing and since $\Delta_2 >1$ this shows that there is a $\delta_1 \in (0,1)$ such that $\Delta_n \ge 1+ \delta_1$ for all $n \in \mathbb N_{\ge 2}$.

\medskip
As $\sum_{d \in \mathbb N} L_d =1$, we can find an $N_1 \in \mathbb N_{\ge 2}$ such that $\sum_{d=1}^{N_1-1} L_d + \frac{1}{2}\delta_1 \geq 1$. This leads to
\[ \begin{split}
\Delta_{N_1+1}  & > 1 + \sum_{d=1}^{N_1-1} L_d \Delta_{N_1+1-d} \\
& \geq 1 + (1+\delta_1) \sum_{d=1}^{N_1-1} L_d \\			
& \geq 1 +  (1+\delta_1) \left(1-\frac{1}{2}\delta_1\right) = 2 + \frac{1}{2} (\delta_1 - \delta_1^2). \\
\end{split}\]
Set $\delta_2 =\frac12 ( \delta_1 - \delta_1^2) \in (0,1)$. Then this shows that $\Delta_n \geq 2 + \delta_2$ for all $n\ge N_1+1$. Continuing in this manner, inductively set
\[\delta_m=\frac{1}{m} (\delta_{m-1} - \delta_{m-1}^2), \quad m \in \mathbb N_{\geq 2}.\] 
Let $M \in \mathbb N_{\ge 2}$ and assume that there exist integers $2 \le N_1 < N_2 < \cdots < N_{M-1}$ such that $\Delta_n \geq m+1 + \delta_{m+1}$ for all $n \ge  N_m + 1$, $1 \le m < M$.  Since $\delta_m \in (0,1)$ for each $m \in \mathbb N$, there exists an $N_{M} > N_{M-1}$ such that 
\[\sum_{d=1}^{N_{M}-N_{M-1}} L_d + \frac{1}{M+1}\delta_{M} \geq 1.\]
This yields 
\[\begin{split}
\Delta_{N_{M}+1}  & >  1 + \sum_{d=1}^{N_{M}-N_{M-1}}  L_d \Delta_{N_{M}+1-d} \\
& \geq 1+   \sum_{d=1}^{N_{M}-N_{M-1}} L_d (M+\delta_{M}) \\	
& \geq 1+ (M+\delta_M) \left(1-\frac1{M+1}\delta_{M}\right) \\				
& = M+1 + \frac{1}{M+1} (\delta_{M} - \delta_{M}^2)\\
& = M+1 + \delta_{M+1}.
\end{split}\]
Hence, we can find a strictly increasing sequence of integers $(N_m)_{m \in \mathbb N}$ with the property that for each $m \in \mathbb N$ we have $\Delta_n > m$ for all $n > N_m$. Thus $\lim_{n \to \infty} \Delta_n = \infty$. 
\end{proof}

\section{$L$-normality}\label{s:normal}
In this section we introduce a set of sequences $A \in \mathcal A^\mathbb N$ for which we prove that they are $L$-normal. These sequences will arise from concatenating the words in the nodes of the $L$-tree depth by depth with a certain multiplicity. We will therefore introduce some notation to regard a sequence of digits in $\mathcal A^\mathbb N$ as a sequence of words in $(\mathcal A^*)^\mathbb N$ as well.

\subsection{Normality of concatenation sequences}
A {\em concatenation map} is a map $c: \mathbb N \to \mathbb N$ that satisfies
\begin{itemize}
\item[\scriptsize $\bullet$] $c(1)=1$,
\item[\scriptsize $\bullet$] $c(i+1)-c(i) \in \{0,1\}$ for all $i \in \mathbb N$,
\item[\scriptsize $\bullet$] $\# \{ i \, : \, c(i)=k \} < \infty$ for each $k \in \mathbb N$.
\end{itemize}
It can be used to split a sequence $A= (a_i)_{i \in \mathbb N} \in \mathcal A^\mathbb N$ into a sequence of words $(\bm{u}_j)_{j \in \mathbb N} \in (\mathcal A^*)^\mathbb N$ by specifying which digits $a_i$ belong to the same word: for each $j \in \mathbb N$ there is an $m(j) = \min\{ i \in \mathbb N \, : \, c(i)=j\}$ and an $M(j) = \max\{ i \in \mathbb N \, : \, c(i)=j\}$, such that $\{ i \in \mathbb N \, : \, c(i)=j\} = \llbracket m(j),M(j) \rrbracket$. Hence, to each $A \in \mathcal A^\mathbb N$ and concatenation map $c$ we can associate a unique sequence $A_c = (\bm{u}_j)_{j \in \mathbb N}$ with $\bm{u}_j = a_{m(j)}\ldots a_{M(j)} \in \mathcal A^*$ for each $j \in \mathbb N$. We refer to the sequence $A_c$ as the {\em concatenation sequence} associated to the pair $(A,c)$. 

\vskip .2cm
For a constant $K_1 \in \mathbb R_{>1}$, we call a sequence $A \in \mathcal A^\mathbb N$ an {\em $(L,K_1)$-tree sequence} if there is a concatenation map $c:\mathbb N \to \mathbb N$ such that the concatenation sequence $A_c = (\bm{u}_j)_{j \in \mathbb N}$ satisfies the following two properties.
\begin{enumerate}
\item[(P1)] For all $n \geq 1$, 
\[ \left\{ j \in \mathbb N \, : \, \bm{u}_j \in G(n) \right\} = \left\llbracket \sum_{m=1}^{n-1} m! +1, \sum_{m=1}^n m!\right\rrbracket.\]
\item[(P2)] For each $n \geq 1$ and $\beta \in G(n)$ there is an error $e_\beta \in (-K_1,K_1)$ such that
\[ \#\left\{1 \le j \leq \sum_{m=1}^n m! \, :\,  \bm{u}_j=\beta  \right\} = n! W(\beta) + e_{\beta}.\]
\end{enumerate}
In other words, (P1) says that the sequence $A_c$ contains precisely $n!$ words of depth $n$ and that all of these words appear before any word of depth $n+1$ appears. (P2) implies that for each $n \in \mathbb N$ and $\beta \in G(n)$ the number of times that $\beta$ occurs in $A_c$ deviates from $n!W(\beta)$ by at most a constant. To show that we can construct sequences satisfying (P2), note that for each $n \in \mathbb N$ and $\beta \in G(n)$ we can start by adding $\lfloor n!W(\beta) \rfloor$ copies of the word $\beta$. Since by Lemma~\ref{l:weight}(i)
\[ n! - 2^n \le \sum_{\beta \in G(n)} \lfloor n!W(\beta) \rfloor \le \sum_{\beta \in G(n)} n! W(\beta) =n!, \]
and $\# G(n)=2^n$, we can achieve that $e_\beta \in [-1,1]$ for each $\beta \in G(n)$ by adding at most $2^n$ words $\beta \in G(n)$ one additional time. Hence, for any constant $K_1 > 1$ sequences satisfying (P1) and (P2) can be constructed.

\medskip
For constants $K_1 \in \mathbb R_{> 1}$ and $K_2 \in \mathbb R_{> 4}$ we call $A \in \mathcal A^\mathbb N$ an {\em $(L,K_1,K_2)$-tree sequence} if it is an $(L,K_1)$-tree sequence that additionally satisfies the following property.
\begin{enumerate}
\item[(P3)] For each $n \in \mathbb N_{\ge 3}, k \in \llbracket 1,n(n-1) \rrbracket$ and $\beta \in G(n-2)$, there is an error $e_\beta^k \in (-K_2,K_2)$ such that
\[ \begin{split}
 \#\Big\{ j \in \Big\llbracket \sum_{m=1}^{n-1} m! + (n-2)!(k-1) +1, \sum_{m=1}^{n-1} m! + (n-2)! k \Big\rrbracket\, :\,  & \, \bm{u}_j^{-2} =\beta\Big\}\\
 = \ & (n-2)!W(\beta) + e^k_{\beta}.
 \end{split}\]
\end{enumerate}
Property (P3) says that if we were to split the words of depth $n$ in $A_c$ into $n(n-1)$ groups of $(n-2)!$ words in order of appearance, then in each of these groups the words appear in such a way that their grandparents satisfy (P2) with error margin $K_2$ instead of $K_1$. Note that this is independent of how the words from $G(n-2)$ in the sequence $A_c$ were actually chosen. In each group the division of the words over the grandparents can be different as long as (P2) is respected with error margin $K_2$.

\medskip
To show that it is possible to achieve (P1), (P2) and (P3) simultaneously, fix an $n \ge 3$ and a $\beta \in G(n-2)$. To satisfy (P1) and (P2) we proceed as above and add each word  $\bm{u} \in \{ \beta11, \beta2, \beta^+1, (\beta^+)^+\}$ either $\lfloor W(\bm{u})n! \rfloor$ or $\lceil W(\bm{u})n! \rceil$ times to $A_c$ so that we add a total of $n!$ words from $G(n)$. Note that
\[ n(n-1)\lfloor W(\bm{u})(n-2)! \rfloor \le \lfloor W(\bm{u})n! \rfloor \le \lceil W(\bm{u})n! \rceil \le n(n-1)\lceil W(\bm{u})(n-2)! \rceil .\]
Since
\[ \sum_{\bm{u}\in G(n)\, : \, \bm{u}^{-2}=\beta} \lfloor (n-2)!W(\bm{u})\rfloor > \sum_{\bm{u}\in G(n)\, : \, \bm{u}^{-2}=\beta} ((n-2)!W(\bm{u})-1) = (n-2)!W(\beta)-4,
\]
and similarly
\[ \sum_{\bm{u}\in G(n)\, : \, \bm{u}^{-2}=\beta} \lceil (n-2)!W(\bm{u})\rceil < (n-2)!W(\beta)+4,
\]
we can spread the available words from $\{ \beta11, \beta2, \beta^+1, (\beta^+)^+\}$ across the $n(n-1)$ groups of $(n-2)!$ words in such a way that the number of words with grandparent $\beta$ in each group lies between $(n-2)!W(\beta)-4$ and $(n-2)!W(\beta)+4$. This shows that for any $K_1 \in \mathbb R_{> 1}$ and $K_2 \in \mathbb R_{> 4}$ we can construct an $(L,K_1,K_2)$-tree sequence.

\medskip
For $A\in \mathcal A^\mathbb N$, a concatenation map $c$ and a word $\alpha \in \mathcal A^*$, let
\begin{equation}\label{q:acalpha}
A_c(\alpha) = \{ j \in \mathbb N \, : \, \bm{u}_j =\alpha \}
\end{equation}
and for each $\alpha \in \mathcal A^*$ and $n \ge 1$ let
\begin{equation}\label{q:acalphan} A_c(\alpha,n) = \{ j \in \mathbb N \, : \, \bm{u}_j \in G(n) \text{ and } \bm{u}_j = \alpha \gamma \text{ for some } \gamma \in \mathcal A^*\}
\end{equation}
be the set of indices of words in $A_c$ that have depth $n$ and have $\alpha$ as a prefix. The next proposition on $(L,K_1)$-tree sequences $A$ states that for any $\alpha \in \mathcal A^*$ the proportion of words in $A_c$ that have $\alpha$ as a prefix asymptotically equals $\mu(\alpha)$ and in that sense the lemma indicates a form of $L$-normality of the concatenation sequence $A_c$.

\begin{prop}\label{genwoord}
Fix some $K_1 \in \mathbb R_{> 1}$.  Let $\alpha \in \mathcal A^*$ and $\varepsilon >0$. Then there is an $N \in \mathbb N$ such that for all $n \ge N$ and any $(L,K_1)$-tree sequence $A \in \mathcal A^\mathbb N$,
\[\left| \frac{\#A_c(\alpha,n)}{n!} - \mu(\alpha) \right| < \varepsilon.\]
\end{prop}

\begin{proof}
Note that by \eqref{q:Ds} it follows that if $\bm{u}_j=\alpha \gamma \in G(n)$, then $D(\gamma) = n-D(\alpha)$. Hence, for each $n > D(\alpha)$ and any $L$-tree sequence $A$ this yields 
\[ \#A_c(\alpha,n) = \sum_{\gamma \in G(n-D(\alpha))} \#A_c(\alpha\gamma). \]
Let $A$ be any $L$-tree sequence. By (P1) we know that all $j$ with $\bm{u}_j = \alpha \gamma \in G(n)$ satisfy $ j \in \llbracket \sum_{m=1}^{n-1} m! + 1, \sum_{m=1}^n m! \rrbracket$. Then from (P2) and Lemma~\ref{l:weight} we obtain
\begin{equation}\label{q:inlemma24}
\begin{split}
\#A_c(\alpha,n)  =\ & \sum_{\gamma \in G(n-D(\alpha))} \#A_c(\alpha\gamma) \\
=\ & \sum_{\gamma \in G(n-D(\alpha))} (n!W(\alpha\gamma) + e_{\alpha\gamma}) \\
=\ & n!  W(\alpha1) \sum_{\gamma \in G(n-D(\alpha))} W(\gamma)+\sum_{\gamma \in G(n-D(\alpha))} e_{\alpha\gamma}  \\
=\ & n! \mu(\alpha) +\sum_{\gamma \in G(n-D(\alpha))} e_{\alpha\gamma} .
\end{split}
\end{equation}
Since 
\begin{equation*}\
\sum_{\gamma \in G(n-D(\alpha))} e_{\alpha\gamma} \in (-K_12^n,K_12^n),
\end{equation*}
and
\[\lim_{n \to \infty} \frac{2^n}{n!} = 0,\]
the result follows.
\end{proof}

\subsection{Normality along the depth subsequence}
For each $n \in \mathbb N$ a word $\bm{u} \in G(n)$ has length $|\bm{u}| \in \llbracket 1, n\rrbracket$. In this subsection we first study for $(L,K_1)$-tree sequences $A \in \mathcal A^\mathbb N$ the average word length of words in $A_c$, counted with multiplicity. For each $n \in \mathbb N$ write
\[ d(n) :=  \sum_{j: \bm{u}_j \in G(n)} |\bm{u}_j|\]
for the total number of digits in $A$ that belong to words from $G(n)$. From (P2) we see that
\[ d(n) =  n!  \sum_{\bm{u} \in G(n)} W(\bm{u}) |\bm{u}|+ \sum_{\bm{u} \in G(n)}e_{\bm{u}}|\bm{u}|.\]
If we put $\hat e_n := \sum_{\bm{u} \in G(n)}e_{\bm{u}}|\bm{u}| \in (-nK_1 2^n,nK_12^n)$, then
\begin{equation}\label{ster}
\frac{d(n)+\hat e_n}{n!} =  \sum_{\bm{u} \in G(n)} W(\bm{u}) |\bm{u}|.
\end{equation}
Note that the different values of $d(n)$ for different $(L,K_1)$-tree sequences $A$ are captured by the corresponding $\hat e_n$. Since $\lim_{n \to \infty} \frac{n2^n}{n!}  =0$, it follows from Lemma~\ref{l:unbounded} that 
 $\lim_{n \to \infty} \frac{d(n)}{n!} = \infty$ uniformly in $A$ in the sense that if we fix a $K_1 \in \mathbb R_{>1}$, then for any $M \in \mathbb N$ there is an $N \in \mathbb N$ such that for all $(L,K_1)$-tree sequences $A \in \mathcal A^\mathbb N$ and all $n \ge N$,
 \begin{equation}\label{q:dnn!}
 \frac{d(n)}{n!} > M.
 \end{equation}

\medskip
By (P1), the total number of words in $A_c$ that have depth at most $n$ is
\[ \# \left\{ j \in \mathbb N \, : \, \bm{u}_j \in \bigcup_{m=1}^n G(n) \right\} = \sum_{m=1}^n m!.\]
For $n \ge 1,$ set $d^*(n) = \sum_{m=1}^n d(m)$. The cumulative average number of digits per word in $A_c$ up to depth $n$ is then given by
\[ \frac{d^*(n)}{\sum_{m=1}^n m!}.\]
Observe that for each $n \in \mathbb N$ it holds that $n!= n \cdot (n-1)! \geq \sum_{m=1}^{n-1} m!,$ thus
\[ \frac{d^*(n)}{\sum_{m=1}^n m!} \ge \frac{d(n)}{2n!}.\]
From \eqref{q:dnn!} we then obtain that in the limit this ratio is unbounded, which is stated in the following lemma.

\begin{lemma}\label{gengemiddelde} Fix a $K_1 \in \mathbb R_{>1}$. For any $M \in \mathbb N$ there is an $N \in \mathbb N$ such that for all $(L,K_1)$-tree sequences $A \in \mathcal A^\mathbb N$ and all $n \ge N$,
\[ \frac{d^*(n)}{\sum_{m=1}^n m!}>M.\]
\end{lemma}

Given a word $\alpha \in \mathcal A^*$, instead of computing the proportion of indices $i$ that mark the start of the occurrence of the words $\alpha$ in $A$, we start by concerning ourselves with the occurrences of the word $\alpha$ as a subword of the words $\bm{u}_j$ of $A_c$. We introduce some further notation, basically to transfer the notation introduced above for the elements $\bm{u}_j$ of the sequence $A_c$ to the elements $a_i$ of the sequence $A$.

\medskip
For each $i \in \mathbb N$ let $m_i$ and $M_i$ be the least and greatest index for which $c(i)=c(m_i)=c(M_i)$, respectively, and set
\[ \begin{split} 
		p_c(i)&:=a_{m_i}...a_{i-1} \in \mathcal A^*_\epsilon, \\
		s_c(i)&:=a_i...a_{M_i} \in \mathcal A^*.
\end{split}\]
So, $\bm{u}_{c(i)} = p_c(i)s_c(i)$. Write $D_c(i) = D(\bm{u}_{c(i)})$ for the depth of the word to which the digit $a_i$ belongs. Note that a word $\alpha \in \mathcal A^*$ can only occur as a subword of a word $\bm{u}_j$ if $D(\bm{u}_j) \ge D(\alpha)$. For each $1 \le  l \le n$ let
\[ \begin{split} 
U_A(n,l):=\ & \{i \in \mathbb N\, :\,  D_c(i)=n, \, D(p_c(i)) < n - l \}\\
=\ & \bigcup_{\substack{\beta \in \mathcal A^*_\epsilon: \\ D(\beta)<n-l}} \{i \in \mathbb N\, :\,  D_c(i)=n, \, p_c(i)=\beta\}
\end{split}\]
be the set of those indices $i$ that correspond to words of depth $n$ and are such that the word $s_c(i)$, starting at index $i$, has depth at least $l$. If we take $l=D(\alpha)$ for some $\alpha \in \mathcal A^*$, then $U_A(n,D(\alpha))$ contains those indices $i$ for which it is possible for $\alpha$ to occur as a subword of $\bm{u}_{c(i)}$ at the position of the digit $a_i$. The set of those indices $i$ for which this actually happens is given by the following set. For $\alpha \in \mathcal A^*$ let $\mathcal A(\alpha)$ be as defined in \eqref{q:cylinder} and let
\[ \begin{split}
U_{A, \alpha}(n):=\ & \{ i \in U_A(n, D(\alpha))\, :\,  s_c(i) \in \mathcal A(\alpha)\}\\
=\ & \bigcup_{\substack{\beta \in \mathcal A^*_\epsilon : \\ D(\beta)<n-D(\alpha)}}\{i \in \mathbb N\, : \,  D_c(i)=n, \, p_c(i)=\beta, \,  s_c(i) \in \mathcal A(\alpha). \}
\end{split}\]
%be the collection of those indices $i$ for which $\bm{u}_{c(i)} \in G(n)$ and $\alpha$ is a subword of $\bm{u}_{c(i)}$ that starts at the position of $a_i$. 
The next lemma states that among the those indices $i$ that still allow enough room within their respective words $\bm{u}_{c(i)}$ for $\alpha$ to appear, i.e., indices $i \in U_A(n,D(\alpha))$, the proportion of indices $i$ where $\alpha$ actually occurs as a subword of $\bm{u}_{c(i)}$ at the position of $a_i$ tends to $\mu(\alpha)$ as $n \to \infty$.

\begin{lemma}\label{gencarol}
Fix a $K_1 \in \mathbb R_{>1}$. Let $\alpha \in \mathcal A^*$ and $\varepsilon >0$. Then there is an $N \in \mathbb N$ such that for any $(L,K_1)$-tree sequence $A \in \mathcal A^\mathbb N$ and all $n \ge N$,
\[ | N_A (\alpha, U_A(n,D(\alpha))) - \mu(\alpha) | < \varepsilon.\]
\end{lemma}

\begin{proof}
For $\alpha, \beta \in \mathcal A^*$ and any $n \in \mathbb N$ let
\[ \begin{split} 
A(\beta,n) &= \{i \in \mathbb N: p_c(i)=\beta,  D_c(i)=n \}, \\
A(\beta,  \alpha,  n) &= \{i \in \mathbb N\, :\,  p_c(i)=\beta, \,  s_c(i) \in \mathcal A(\alpha),  \, D_c(i)=n \},
\end{split}\]
so that 
\[  U_A(n,l) =\bigcup_{\substack{\beta \in \mathcal A^*_\epsilon : \\ D(\beta)<n-l}} A(\beta,n) \quad \text{ and } \quad
U_{A, \alpha }(n) =\bigcup_{\substack{\beta \in \mathcal A^*_\epsilon: \\ D(\beta)<n-D(\alpha)}}A(\beta,\alpha,n).\]
Recall the definition of $A_c(\beta,n)$ from \eqref{q:acalphan} and note that to each $j \in A_c(\beta,n)$ there corresponds a unique $i \in A(\beta,n)$ with $c(i)=j$ and vice versa. Hence
\[  \# A(\beta,n) = \# A_c(\beta,n),\]
and similarly,
\[ \# A(\beta,\alpha, n) =  \# A_c(\beta\alpha,n) .\]
%Lemma~\ref{genwoord} then yields
%\begin{equation}\label{q:genklassen}
%\lim_{n \to \infty}\frac{\#A(\beta,\alpha,n)}{\#A(\beta,n)}=\lim_{n \to \infty} \frac{\#A_c(\beta\alpha,n)}{\#A_c(\beta,n)} = \frac{\mu(\beta\alpha)}{\mu(\beta)} =  \mu(\alpha).
%\end{equation}
Note that both families of sets $\{U_A(n,D(\alpha))\}_{n \in \mathbb N}$ and $\{U_{A, \alpha}(n)\}_{n \in \mathbb N}$ are pairwise disjoint. Then for any $n > D(\alpha)$ by \eqref{q:inlemma24} it follows that
\begin{equation*}
\begin{split}
N_A (\alpha, U_A(n,D(\alpha)))&= \frac{\#U_{A, \alpha}(n)}{\#U_A(n,D(\alpha))} = \frac{\sum_{\substack{\beta \in \mathcal A^*_\epsilon : \\ D(\beta)<n-D(\alpha)}} \#A(\beta,\alpha,n)}{\sum_{\substack{\beta \in \mathcal A^*_\epsilon: \\ D(\beta)<n-D(\alpha)}}\#A(\beta,n)}\\
& =\frac{\sum_{\substack{\beta \in \mathcal A^*_\epsilon : \\ D(\beta)<n-D(\alpha)}}  \left( n! \cdot \mu(\beta \alpha)+ \sum_{\gamma \in G(n-D(\beta\alpha))}e_{\beta \alpha \gamma} \right)  }{\sum_{\substack{\beta \in \mathcal A^*_\epsilon : \\ D(\beta)<n-D(\alpha)}}\left(n! \cdot \mu(\beta)+\sum_{\gamma \in G(n-D(\beta))}e_{\beta\gamma}\right)}\\
& =\mu(\alpha) \frac{\sum_{\substack{\beta \in \mathcal A^*_\epsilon : \\ D(\beta)<n-D(\alpha)}}  \left( \mu(\beta)+ \frac1{\mu(\alpha)n!}\sum_{\gamma \in G(n-D(\beta\alpha))}e_{\beta \alpha \gamma} \right)  }{\sum_{\substack{\beta \in \mathcal A^*_\epsilon : \\ D(\beta)<n-D(\alpha)}}\left(\mu(\beta)+\frac1{n!}\sum_{\gamma \in G(n-D(\beta))}e_{\beta\gamma}\right)}.
	\end{split}
\end{equation*}
Note that for any $n > D(\alpha)$,
\[ 0< L_1 = \mu(1)  \le \sum_{\substack{\beta \in \mathcal A^*_\epsilon : \\ D(\beta)<n-D(\alpha)}}  \mu(\beta) \le n,\]
and
\[\sum_{\substack{\beta \in \mathcal A^*_\epsilon: \\ D(\beta)<n-D(\alpha)}} \sum_{\gamma \in G(n-D(\beta\alpha))} e_{\beta \alpha \gamma}, \sum_{\substack{\beta \in \mathcal A^*_\epsilon: \\ D(\beta)<n-D(\alpha)}} \sum_{\gamma \in G(n-D(\beta))}e_{\beta\gamma} \in (-K_1^24^n,K_1^24^n).\]
Therefore,
\[ \frac{\sum_{\substack{\beta \in \mathcal A^*_\epsilon : \\ D(\beta)<n-D(\alpha)}}  \left( \mu(\beta)+ \frac1{\mu(\alpha)n!}\sum_{\gamma \in G(n-D(\beta\alpha))}e_{\beta \alpha \gamma} \right)  }{\sum_{\substack{\beta \in \mathcal A^*_\epsilon : \\ D(\beta)<n-D(\alpha)}}\left(\mu(\beta)+\frac1{n!}\sum_{\gamma \in G(n-D(\beta))}e_{\beta\gamma}\right)} \in \left(  \frac{1-\frac1{\mu(\alpha)}\frac{K_1^24^n}{n! L_1}}{1+ \frac{K_1^24^n}{n!L_1}},  \frac{1+\frac1{\mu(\alpha)}\frac{K_1^24^n}{n! L_1}}{1- \frac{K_1^24^n}{n!L_1}} \right),
\]
and the result follows.
\end{proof}

%\begin{remark}\label{tja}
%Observe that in our construction we now use $n!$ instead of $sf(n)$,  thus decreasing the amount of words added.  There are slower growing sequences available such as $k_n=5^n$,  since we only need $\frac{4^n}{k_n}$ to go to zero in the limit, as the above shows.  We chose $n!$ because it is so similar to $sf(n),$ so that some new definitions resemble the old approach as much as possible.
%\end{remark}

For any $n,l \in \mathbb N$ set $U_A^*(n,l) = \bigcup_{m=1}^n U_A(m,l)$. The next step is to extend Lemma~\ref{gencarol} to $U_A^*(n,D(\alpha))$. %Note that, again using that $U_A(n,D(\alpha)) \cap U_A(m,D(\alpha)) = \emptyset$ if $n \neq m$, we have
%\begin{equation*}
%	\begin{split}
%	N_A\left(\alpha, U_A^*(n,D(\alpha))\right) & = \frac{ \sum_{m=1}^n\#\{i \in U_A(m,D(\alpha)): a_i=\alpha_1,  \dots ,  a_{i+k-1}= \alpha_k\}}{\sum_{m=1}^n\#U_A(m,D(\alpha))} \\
%				& = \frac{ \sum_{m=1}^nN_A(\alpha,U_A(m,D(\alpha))) \cdot \#U_A(m,D(\alpha))}{\sum_{m=1}^n\#U_A(m,D(\alpha))}. \\
%	\end{split}
%\end{equation*}

\begin{prop}\label{p:NUstar}
Fix $K_1 \in \mathbb R_{>1}$. Let $\alpha \in \mathcal A^*$ and $\varepsilon \in (0, \mu(\alpha))$. Then there is an $N \in \mathbb N$ such that for all $(L,K_1)$-tree sequences $A \in \mathcal A^\mathbb N$ and all $n \ge N$,
\[ | N_A\left(\alpha, U_A^*(n,D(\alpha))\right) - \mu(\alpha)| < \varepsilon.\]
\end{prop}

\begin{proof}
Fix $\varepsilon \in (0, \mu(\alpha))$. By Lemma~\ref{gencarol} we know that there exists an $N_1 \in \mathbb N$ such that for all $m \ge N_1$ and each $A$ we have
\[|\mu(\alpha) - N_A(\alpha,U_A(m,D(\alpha)))| < \frac{1}{2}\varepsilon.\]
Note that for each $(L,K_1)$-tree sequence $A$ we have $\#U_A(m,D(\alpha)) \geq 1$ for all $m > D(\alpha)$. Hence, there is an $N_2 > N_1$ such that for all $n \ge N_2$ and all $(L,K_1)$-tree sequences $A$, 
\[0<\frac{\sum_{m=1}^{N_1-1}\#U_A(m,D(\alpha))}{\sum_{m=N_1}^n \#U_A(m,D(\alpha))}< \frac{1}{2}\varepsilon.\]
Since $N_A(\alpha,U_A(m,D(\alpha))) \in [0,1]$ for all $m \in \mathbb N$ and $A$ and again using that \[U_A(n,D(\alpha)) \cap U_A(m,D(\alpha)) = \emptyset, \quad n \neq m,\]
we have the following estimate for all $n \geq N_2$ and $A$:
\begin{equation*}
\begin{split}
N_A\left(\alpha, U_A^*(n,D(\alpha))\right)	& = \frac{ \sum_{m=1}^n\#\{i \in U_A(m,D(\alpha)): s_c(i)\in \mathcal A(\alpha)\}}{\sum_{m=1}^n\#U_A(m,D(\alpha))} \\
& = \frac{ \sum_{m=1}^{n}N_A(\alpha,U_A(m,D(\alpha))) \cdot \#U_A(m,D(\alpha))}{\sum_{m=1}^{n}\#U_A(m,D(\alpha))}\\
&\leq \frac{ \sum_{m=1}^{N_1-1} \#U_A(m,D(\alpha))+\sum_{m=N_1}^{n}(\mu(\alpha)+\frac{1}{2}\varepsilon) \cdot \#U_A(m,D(\alpha))}{\sum_{m=1}^{n}\#U_A(m,D(\alpha))}\\
&\leq \frac{ \sum_{m=1}^{N_1-1} \#U_A(m,D(\alpha))}{\sum_{m=N_1}^{n}\#U_A(m,D(\alpha))} + \frac{\sum_{m=N_1}^{n}(\mu(\alpha)+\frac{1}{2}\varepsilon) \cdot \#U_A(m,D(\alpha))}{\sum_{m=N_1}^{n}\#U_A(m,D(\alpha))}\\
&< \frac{1}{2}\varepsilon+ \mu(\alpha) + \frac{1}{2}\varepsilon \\
&=\mu(\alpha) + \varepsilon.
\end{split}
\end{equation*}
Similarly, we obtain
\begin{equation*}
	\begin{split}
	N_A\left(\alpha, U_A^*(n,D(\alpha))\right)
				&\geq \frac{\sum_{m=N_1}^{n}(\mu(\alpha)-\frac{1}{2}\varepsilon) \cdot \#U_A(m,D(\alpha))}{\sum_{m=1}^{n}\#U_A(m,D(\alpha))}\\
				&= \frac{\sum_{m=N_1}^{n}(\mu(\alpha)-\frac{1}{2}\varepsilon) \cdot \#U_A(m,D(\alpha))}{\sum_{m=N_1}^{n}\#U_A(m,D(\alpha))} \cdot \frac{\sum_{m=N_1}^{n}\#U_A(m,D(\alpha))}{\sum_{m=1}^{n}\#U_A(m,D(\alpha))}\\
				&= \Big(\mu(\alpha)-\frac{1}{2}\varepsilon \Big) \cdot \left(1- \frac{\sum_{m=1}^{N_1-1}\#U_A(m,D(\alpha))}{\sum_{m=1}^{n}\#U_A(m,D(\alpha))}\right) \\
				&\geq \Big(\mu(\alpha)-\frac{1}{2}\varepsilon \Big) \cdot \Big(1-\frac{1}{2}\varepsilon \Big)  \\
&\geq \mu(\alpha) - (1+ \mu(\alpha)) \frac{1}{2}\varepsilon \\
&\geq \mu(\alpha) - \varepsilon.
	\end{split}
\end{equation*}
This gives the result.
\end{proof}

The next result compares the number of  indices in $U_A^*(n,D(\alpha))$ to the total number of indices corresponding to digits in words up to depth $n$, which is $d^*(n)$. 

\begin{lemma}\label{genverhouding}
Fix a $K_1 \in \mathbb R_{>1}$. Let $\alpha \in \mathcal A^*$ and $\varepsilon>0$. Then there is an $N_1 \in \mathbb N$ such that for all $n \ge N_1$ and all $(L,K_1)$-tree sequences $A  \in \mathcal A^\mathbb N$,
\[ \frac{\# U_A(n,D(\alpha))}{ d(n)} \in (1-\varepsilon,1]\] 
and there is an $N_2 \in \mathbb N$ such that for all $n \ge N_2$ and all $(L,K_1)$-tree sequences $A \in \mathcal A^\mathbb N$,
\[\frac{\# U_A^*(n,D(\alpha))}{ d^*(n)} \in (1-\varepsilon,1].\] 
\end{lemma}

\begin{proof}
For any $\bm{u} \in \mathcal A^*$ it holds that $|\bm{u}| \le D(\bm{u})$. Hence, by \eqref{q:Ds}, for any $(L,K_1)$-tree sequence $A$ and index $i \in \mathbb N$ it holds that
\[ |s_c(i)| \le D_c(i)-D(p_c(i)).\]
This implies that if $|s_c(i)| \ge D(\alpha)+1$, then $i \in U_A(n, D(\alpha))$ for any $n \ge D_c(i)$ or in other words, at most $D(\alpha)$ digits $a_i$ of each word $\bm{u}_{c(i)}$ in $A_c$ with $D(\bm{u}_{c(i)})=n$ do not belong to $U_A(n,D(\alpha))$. This yields
\begin{equation*}\label{q:genpiet}
\# U_A(n,D(\alpha)) \ge d(n)-D(\alpha)n!.
\end{equation*}
Let $\varepsilon >0$. Then by \eqref{q:dnn!} there exists an $N_1 \in \mathbb N$ such that for all $n \ge N_1$ and any $(L,K_1)$-tree sequence $A$,
\[ 1 \geq  \frac{\#  U_A(n,D(\alpha))}{ d(n)}  \geq \frac{ d(n)-D(\alpha)n!}{ d(n)} = 1- D(\alpha)\frac{n!}{ d(n)} > 1-\varepsilon.\]
Similarly, from Lemma~\ref{gengemiddelde} we obtain an $N_2 \in \mathbb N$ such that for all $n \ge N_2$ and any $(L,K_1)$-tree sequence $A$,
\[ 1 \geq  \frac{\#  U_A^*(n,D(\alpha))}{ d^*(n)} \ge \frac{d^*(n)-D(\alpha)\sum_{m=1}^n m!}{d^*(n)}  >1-\varepsilon.\qedhere \]
\end{proof}

%\begin{cor}\label{lokaal}
%Let $\alpha \in \mathcal A^*$ and let $A  \in \mathcal A^\mathbb N$ be a sequence that allows a concatenation map $c: \mathbb N \to \mathbb N$ such that (P1) and (P2) are satisfied. Then
%\[\lim_{n \to \infty} \frac{\#U_A(n,D(\alpha))}{d(n)}=1.\]
%\end{cor}

With this lemma we can extend the result from Proposition~\ref{p:NUstar} from $ U_A^*(n, D(\alpha))$ to all of $\llbracket 1,  d^*(n) \rrbracket$.
\begin{prop}\label{p:d-normal}
Fix a $K_1 \in \mathbb R_{>1}$. Let $\alpha \in \mathcal A^*$ and $\varepsilon>0$. Then there is an $N_1 \in \mathbb N$ such that for all $n \ge N_1$ and $(L,K_1)$-tree sequences $A = (a_i)_{i \in \mathbb N} \in \mathcal A^\mathbb N$,
\[ |N_A(\alpha, \llbracket d^*(n-1)+1, d^*(n) \rrbracket ) - \mu(\alpha)| < \varepsilon\]
and there is an $N_2 \in \mathbb N$ such that for all $n \ge N_2$ and $(L,K_1)$-tree sequences $A = (a_i)_{i \in \mathbb N} \in \mathcal A^\mathbb N$,
\[ |N_A(\alpha, \llbracket 1, d^*(n) \rrbracket ) - \mu(\alpha)| < \varepsilon.\]
\end{prop}

\begin{proof}
Write $|\alpha|=k$. Then
\[ \begin{split}
N_A (\alpha,  \llbracket d^*(n-1)+1&,   d^*(n) \rrbracket )\\
=\ & \frac{\# \{ i \in  U_A(n, D(\alpha)) \, : \,   a_i = \alpha_1, \ldots, a_{i+k-1}=\alpha_k \}}{ d(n)}\\
& + \frac{\# \{ i \in \llbracket d^*(n-1)+1, d^*(n) \rrbracket \setminus U_A(n, D(\alpha)) \, : \, a_i = \alpha_1, \ldots, a_{i+k-1}=\alpha_k \}}{ d(n)}\\
=\ & \frac{N_A(\alpha,  U_A(n, D(\alpha))) \cdot \#  U_A(n, D(\alpha))}{d(n)}\\
& + \frac{N_A(\alpha, \llbracket d^*(n-1)+1,  d^*(n) \rrbracket \setminus U_A(n, D(\alpha))) \cdot (  d(n) - \#  U_A(n, D(\alpha)))}{ d(n)}.\\
\end{split}\]

Let $\varepsilon >0$. By Lemma~\ref{gencarol} and Lemma~\ref{genverhouding} there is an $N_1 \in \mathbb N$ such that for each $n \ge N_1$ and each $(L,K_1)$-tree sequence $A$,
\[ 1-\frac{\varepsilon}{3} < \frac{\# U_A(n, D(\alpha))}{d(n)} \le 1\]
and
\[ |N_A (\alpha, U_A(n,D(\alpha)))-\mu(\alpha)| < \frac{\varepsilon}{3}.\]
Then
\[ \begin{split} 
|N_A(\alpha, \llbracket d^*(n-1)+1,&  d^*(n) \rrbracket ) - \mu(\alpha)|\\
\le \ & |N_A(\alpha, U_A(n,D(\alpha)))-\mu(\alpha)| \frac{\# U_A(n,D(\alpha))}{d(n)} + \mu(\alpha) \left| \frac{\# U_A(n,D(\alpha))}{d(n)} -1 \right|\\
& + N_A(\alpha, \llbracket d^*(n-1)+1,  d^*(n) \rrbracket \setminus U_A(n, D(\alpha))) \cdot \frac{  d(n) - \#  U_A(n, D(\alpha))}{ d(n)}\\
< & \frac{\varepsilon}{3} + \mu(\alpha)\frac{\varepsilon}{3} + \frac{\varepsilon}{3} < \varepsilon.
\end{split}\]

\medskip
Similarly,
\[ \begin{split}
N_A (\alpha, \llbracket 1,  d^*(n) \rrbracket ) =\ & \frac{\# \{ i \in  U_A^*(n, D(\alpha)) \, : \,   a_i = \alpha_1, \ldots, a_{i+k-1}=\alpha_k \}}{ d^*(n)}\\
& + \frac{\# \{ i \in \llbracket 1, d^*(n) \rrbracket \setminus U_A^*(n, D(\alpha)) \, : \, a_i = \alpha_1, \ldots, a_{i+k-1}=\alpha_k \}}{ d^*(n)}\\
=\ & \frac{N_A(\alpha,  U_A^*(n, D(\alpha))) \cdot \#  U_A^*(n, D(\alpha))}{d^*(n)}\\
& + \frac{N_A(\alpha, \llbracket 1,  d^*(n) \rrbracket \setminus U_A^*(n, D(\alpha))) \cdot (  d^*(n) - \#  U_A^*(n, D(\alpha)))}{ d^*(n)}.\\
\end{split}\]
The result now follows as above from Lemma~\ref{genverhouding} and Proposition~\ref{p:NUstar}.
\end{proof}

Proposition~\ref{p:d-normal} states that $(L,K_1)$-tree sequences $A \in \mathcal A^\mathbb N$ have normality along the subsequence of indices $(d^*(n))_{n \in \mathbb N}$. The additional assumption (P3) aims to control what happens for the intermediate indices $ i \in \llbracket d^*(n-1)+1, d^*(n) \rrbracket$. In the next section we prove that this is enough to obtain $L$-normal sequences.

\subsection{$L$-normal sequences}
Let $A = (a_i)_{i \in \mathbb N} \in \mathcal A^\mathbb N$ be an $(L,K_1,K_2)$-tree sequence for some $K_1 \in \mathbb R_{>1}$ and $K_2 \in \mathbb R_{>4}$ and with concatenation map $c$. The digits $a_i$ with indices $i \in \llbracket d^*(n-1)+1, d^*(n) \rrbracket$ correspond to all $n!$ words $\bm{u}_{c(i)} \in G(n)$ that appear in $A_c$. The assumption of property (P3) regards a division of the indices in the interval $\llbracket d^*(n-1)+1, d^*(n) \rrbracket$ into $n(n-1)$ intervals of indices that each correspond to $(n-2)!$ words from $A_c$. For any $n \in \mathbb N_{\ge 3}$ and $k \in \llbracket 1, n(n-1) \rrbracket$ let
\[ I_A^k(n) : = \left\{ i \in \mathbb N\, :\,  c(i) \in \Big\llbracket \sum_{m=1}^{n-1}m! +1 + (n-2)!(k-1) ,\sum_{m=1}^{n-1}m! + (n-2)!k \Big\rrbracket  \right\}\]
be the set of corresponding indices in $A$.

\begin{prop}\label{opa}
Fix $\alpha \in \mathcal A^*$ and let $A$ be an $(L,K_1,K_2)$-tree sequence. Then for all $\varepsilon >0$ there is an $N \in \mathbb N$ such that for all $n \ge N$ and $k \in \llbracket 1, n(n-1) \rrbracket$,
\[ |N_{A}(\alpha,  I^k_{A}(n))-\mu(\alpha)| < \varepsilon. \]
\end{prop}

\begin{proof}
For each $n \in \mathbb N_{\ge 3}$ and $k \in \llbracket 1, n(n-1) \rrbracket$ consider the words
\[ \bm{u}_{\sum_{m=1}^{n-1}m!+(k-1)(n-2)!+j}, \quad j \in \llbracket 1, (n-2)! \rrbracket.\]
Property (P3) implies that the grandparents of these words satisfy (P2) with error margin $K_2$. Therefore, we can construct an $(L,K_2)$-tree sequence $A' = A'(n,k) \in \mathcal A^\mathbb N$ 
for which the concatenation sequence $A'_{c'}=(\bm{u}_j')_{j \in \mathbb N}$
has 
\[ \bm{u}'_{\sum_{m=1}^{n-3}m! +j} = \bm{u}^{-2}_{\sum_{m=1}^{n-1}m!+(k-1)(n-2)!+j}, \quad j \in \llbracket 1, (n-2)! \rrbracket.\]
For the error margins this implies that $e'_\beta = e^k_\beta$ for each $\beta \in G(n-2)$. Since for each word $\bm{u}$ we have $|\bm{u}^{-2}|\le |\bm{u}| \le |\bm{u}^{-2}|+2$, it then holds that
\begin{equation}\label{q:d'I}
d'(n-2) \le \# I_A^k(n) \le d'(n-2) + (n-2)! \cdot 2,
\end{equation}
where $d'(n-2) = \sum_{j: \bm{u}'_j \in G(n-2)} |\bm{u}'_j|$.

\medskip
We now compare the possible occurrences of the word $\alpha$ in the string
\begin{equation}\label{q:uchanges1}
\bm{u}'_{\sum_{m=1}^{n-3}m!+1} \ldots \bm{u}'_{\sum_{m=1}^{n-2}m!}
\end{equation}
to those in the string
\begin{equation}\label{q:uchanges2}
\bm{u}_{\sum_{m=1}^{n-1}m!+(k-1)(n-2)!+1} \ldots \bm{u}_{\sum_{m=1}^{n-1}m!+k(n-2)!}.
\end{equation}
Changing a word $\bm{u}'_{\sum_{m=1}^{n-3}m! +j}$ to the corresponding word $\bm{u}_{\sum_{m=1}^{n-1}m!+(k-1)(n-2)!+j}$ amounts to changing or adding at most two digits. Note that a change or addition of a single digit of the sequence $A'$ in a given position $i^*$, say, has an effect on the number of occurrence of the word $\alpha \in A'$ only in the range of indices $\llbracket i^*-|\alpha|+1, i^* \rrbracket$. Hence, the change from \eqref{q:uchanges1} to \eqref{q:uchanges2} can result in a change in at most $2|\alpha|+2$ occurrences of the word $\alpha$. Since $I_A^k(n)$ contains digits belonging to precisely $(n-2)!$ words, we obtain
\[ |N_A(\alpha, I_A^k(n))\# I_A^k(n) - N_{A'} (\alpha, \llbracket (d')^*(n-3)+1, (d')^*(n-2) \rrbracket ) d'(n-2)| \le (n-2)!(2|\alpha|+2).\]
Combining this with \eqref{q:d'I} gives
\[  N_A(\alpha, I_A^k(n)) \le  N_{A'} (\alpha, \llbracket (d')^*(n-3)+1, (d')^*(n-2) \rrbracket )  + (2|\alpha|+2)\frac{(n-2)!}{d'(n-2)}\]
and
\begin{multline*}
N_A(\alpha, I_A^k(n))\\
\ge  N_{A'} (\alpha, \llbracket (d')^*(n-3)+1, (d')^*(n-2) \rrbracket ) \left(\frac1{1+2\frac{(n-2)!}{d'(n-2)}} \right)- (2|\alpha|+2)\frac{\frac{(n-2)!}{d'(n-2)}}{1+2 \frac{(n-2)!}{d'(n-2)}}.
\end{multline*}
The result now follows by applying \eqref{q:dnn!} and Proposition~\ref{p:d-normal} to the sequence $A'$.
\end{proof}

Proposition~\ref{opa} gives the normality of $(L,K_1,K_2)$-tree sequences $A$ along the subsequence of indices marked by the endpoints of the sets $I_A^k(n)$. To prove the $L$-normality of $(L,K_1,K_2)$-tree sequences it remains to consider the indices that fall within the sets $I_A^k(n)$. This is done in the next theorem.

\begin{theorem}
Let $A$ be an $(L,K_1,K_2)$-tree sequence. Then for any $\alpha \in \mathcal A^*$,
\[\lim_{M \to  \infty} N_A(\alpha, \llbracket 1, M \rrbracket) = \mu(\alpha).\]
\end{theorem}

\begin{proof}
For any integer $M \geq d^*(6)$ there are unique integers $n_M \in \mathbb N$ and  $k_M \in \llbracket 1, n_M(n_M-1) \rrbracket$ such that $M\in I_A^{k_M}(n_M)$.  Write
\[R_M=M-d^*(n_M-1) - \sum_{k=1}^{k_M-1} \# I_A^k(n_M) \in \mathbb N.\]
Then
\[ \begin{split}
N_A(\alpha, \llbracket 1, M\rrbracket )  =\ & \frac{N_A(\alpha, \llbracket 1, d^*(n_M-1) \rrbracket)d^*(n_M-1)}{M} + \sum_{k=1}^{k_M-1}\frac{N_A(\alpha, I_A^k(n_M))\#I_A^k(n_M)}{M} \\
& + \frac{N_A(\alpha, \llbracket M-R_M +1, M \rrbracket)R_M}{M}.
\end{split}\]
We begin by showing that $\lim_{M \to \infty} \frac{R_M}M =0$, so that we can disregard the last term. As in the proof of Proposition~\ref{opa}, using (P3) we find an $(L,K_2)$-tree sequence $A'=A'(n_M,k_M) \in \mathcal A^\mathbb N$ for which the concatenation sequence $A'_{c'}=(\bm{u}_j')_{j \in \mathbb N}$
has 
\[ \bm{u}'_{\sum_{m=1}^{n_M-3}m! +j} = \bm{u}^{-2}_{\sum_{m=1}^{n_M-1}m!+(k_M-1)(n_M-2)!+j}, \quad j \in \llbracket 1, (n_M-2)! \rrbracket.\]
If we again let $d'(n-2) = \sum_{j: \bm{u}'_j \in G(n-2)} |\bm{u}'_j|$ for each $n \in \mathbb N$, then $R_M \in \llbracket 1, d'(n_M-2) +  2(n_M-2)! \rrbracket$. Recall the definition of the increasing sequence $(\Delta_n)_{n \in \mathbb N}$ from the proof of Lemma~\ref{l:unbounded}. By \eqref{ster} we have for any $n \in \mathbb N_{\ge 3}$ that
\begin{equation*}
\begin{split}
\frac{d'(n-2)}{d^*(n-1)} \leq \frac{d'(n-2)}{d(n-1)} & \leq \frac{\Delta_{n-2} \cdot (n-2)! + (n-2)K_2 2^{n-2}}{\Delta_{n-1} \cdot (n-1)! - (n-1)K_1 2^{n-1}} \\
& \leq \frac{\frac{\Delta_{n-2}}{n-1} + \frac{K_2 2^{n-2}}{(n-2)!}}{\Delta_{n-1} - \frac{K_1 2^{n-1}}{(n-2)!}} \\
&\leq  \frac{\frac{1}{n-1} + \frac{K_2 2^{n-2}}{\Delta_{n-2}(n-2)!}}{1 - \frac{K_1 2^{n-1}}{\Delta_{n-2}(n-2)!}}.
\end{split}
\end{equation*}
By Lemma~\ref{l:unbounded},
\[ \lim_{n \to \infty} \frac{\frac{1}{n-1} + \frac{K_2 2^{n-2}}{\Delta_{n-2}(n-2)!}}{1 - \frac{K_1 2^{n-1}}{\Delta_{n-2}(n-2)!}}=0.\]
As $\lim_{M \to \infty}n_M = \infty$, by \eqref{q:dnn!} we also have
\[ \lim_{M \to \infty} \frac{(n_M-2)!}{d(n_M-2)}  =0.\]
Since for each $M$,
\[ 0 \le \frac{R_M}{M} \le \frac{d'(n_M-2)}{d^*(N_M-1)} + \frac{2(n_M-2)!}{d(n_M-2)},\]
we find that
\begin{equation}\label{q:MRM}
\lim_{M \to \infty} \frac{R_M}{M} = 0,
\end{equation}
which implies that
\[ \lim_{M \to \infty} \frac{N_A(\alpha, \llbracket M-R_M +1, M \rrbracket)R_M}{M} =0.\]
Let $\varepsilon >0$. Combining Proposition~\ref{p:d-normal}, Proposition~\ref{opa} and \eqref{q:MRM} we know that there exists an $N \in \mathbb N$ such that for all $M \in \mathbb N$ with $n_M \ge N$ it holds that
\[ \begin{split} \frac{N_A(\alpha, \llbracket 1, d^*(n_M-1) \rrbracket)d^*(n_M-1)}{M} & +  \sum_{k=1}^{k_M-1}\frac{N_A(\alpha, I_A^k(n_M))\#I_A^k(n_M)}{M}\\
 \le\ & (\mu(\alpha)+\varepsilon) \frac{d^*(n_M-1) +\sum_{k=1}^{k_M-1} \#I_A^k(n_M) }{M}\\
  \le \ &\mu(\alpha)+\varepsilon
 \end{split} \]
and
\[ \frac{N_A(\alpha, \llbracket 1, d^*(n_M-1) \rrbracket)d^*(n_M-1)}{M} + \sum_{k=1}^{k_M-1}\frac{N_A(\alpha, I_A^k(n_M))\#I_A^k(n_M)}{M} \ge (\mu(\alpha)+\varepsilon) (1-\varepsilon).\]
This yields the desired result.
\end{proof}

Let us give two examples of constructions of $(L,K_1,K_2)$-tree sequences, one for the L\"uroth probability sequence $L = (\frac1{d(d+1)})_{d \in \mathbb N}$ and one for the dyadic probability sequence $L = (\frac1{2^d})_{d \in \mathbb N}$.

\begin{ex}\label{voorbeeld1}
Let $L = (\frac1{d(d+1)})_{d \in \mathbb N}$. We describe the construction of an $(L,1+\varepsilon,4+\varepsilon)$-tree sequence up to the words of depth 4.

\medskip

\noindent \underline{Depth 1:} According to (P1) we add one word of depth 1. As $G(1)=\{ 1\}$, this leaves no choice and we put $a_1 = 1 = \bm{u}_1$. For the concatenation sequence $c$ this implies that $m(1)=M(1)=1$, so $c(1)=1$.

\medskip
\noindent \underline{Depth 2:} Condition (P1) tells us that we add two words of depth 2. Since $G(2) = \{ 11, 2\}$ and $W(11) = \frac12 = W(2)$, it natural to add both words once. Note however that (P2) would also allow either word to be used twice. We choose to put $\bm{u}_2 = 11$ and $\bm{u}_3 = 2$. This yields $a_2=a_3=1$ and $a_4 =2$ and $c(2)=c(3)=2$ and $c(4)=3$.

\medskip
\noindent \underline{Depth 3:} By condition (P1) we have to add six words from the set $G(3) = \{ 111, 12, 21, 3\}$. They have weights
\[W(111)= \Big(\frac12\Big)^2 = \frac14, \quad W(12) = \Big(\frac12\Big)^2 = \frac14, \quad W(21) = \frac12 \cdot \frac13 = \frac16, \quad W(3) = \frac12\cdot \frac23 = \frac13. \]
For $\varepsilon < \frac12$, (P2) forces us to add $111$ and $12$ either once or twice. It is again natural to $21$ once and $3$ twice, but as in the previous step (P2) allows us to give or take at least one. As we still do not have to take (P3) into account, we can choose e.g.~$\bm{u}_4 = 111$, $\bm{u}_5 = 111$, $\bm{u}_6=12$, $\bm{u}_7 = 21$, $\bm{u}_8 = 3$, $\bm{u}_9=3$. This gives
\[ a_5 \cdots a_{16} = 111\, 111\, 12\, 21\, 3\, 3.\]
Hence, $c(5)=c(6)=c(7)=4$, $c(8)=c(9)=c(10)=5$, $c(11)=c(12)=6$, $c(13)=c(14)=7$, $c(15)=8$, $c(16)=9$.

\medskip
\noindent \underline{Depth 4:} Starting with condition (P1), we add 24 words from $G(4) = \{ 1111, 112, 121, 13, 211, 22, 31, 4\}$. The corresponding weights are
\[ \begin{array}{llll}
W(1111) = \frac18, & W(112) = \frac18, & W(121) = \frac1{12}, & W(13) = \frac16,\\
W(211) = \frac1{12}, & W(22) = \frac1{12}, & W(31) = \frac1{12}, & W(4) = \frac14.
\end{array}\]
Although (P2) again leaves some room for variation, we choose to add the words 1111 and 112 three times each, the words 121, 211, 22 and 31 twice each, the word 13 four times and the word 4 six times. We now also have to consider (P3). The words 1111, 112, 121 and 13 have grandparent 11 and the words 211, 22, 31 and 4 have grandparent 2. We divide the 24 words we will add up into twelve groups of two words each. In each of these groups, one word has to come from the set $\{ 1111, 112, 121, 13\}$ and one words has to come from the set $\{211, 22, 31,4\}$. We can take for example
\[ \begin{array}{cc|cc|cc|cc|cc|cc}
\bm{u}_{10} & \bm{u}_{11} & \bm{u}_{12} & \bm{u}_{13} & \bm{u}_{14} & \bm{u}_{14} & \bm{u}_{16} & \bm{u}_{17} & \bm{u}_{18} &
\bm{u}_{19} & \bm{u}_{20} & \bm{u}_{21}\\
1111 & 211 & 1111 & 211 & 1111 & 22 & 112 & 22 & 112 & 31 & 112 & 31\\
%a_{18}a_{19}a_{20}a_{21} & a_{22} a_{23} a_{24} & a_{25} a_{26} a_{27} & a_{28} & a_{29} a_{30} & a_{31} a_{32} a_{33} a_{34} & a_{35} a_{36} & a_{37} a_{38} a_{39} & a_{40} a_{41} a_{42} & a_{43} a_{44} & a_{45} a_{46}\\
\\
\bm{u}_{22} & \bm{u}_{23} & \bm{u}_{24} &
\bm{u}_{25} & \bm{u}_{26} & \bm{u}_{27} &
\bm{u}_{28} & \bm{u}_{29} & \bm{u}_{30} &
\bm{u}_{31} & \bm{u}_{32} & \bm{u}_{33}\\
121 & 4 & 121 & 4 & 13 & 4 & 13 & 4 & 13 & 4 & 13 & 4,
\end{array}\]
which so far yields
\[ \begin{split}
A =\ & 1 \, 11\, 2 \, 111\, 111\, 12\, 21\, 3\, 3 \, 1111\, 211\, 1111 \, 211\, 1111\, 22\, 112\, 22\, 112\, 31\, 112 \, 31\, 121\, 4\, 121\, 4\, 13\\ 
& 4\, 13\, 4\, 13\, 4\, 13\, 4 \cdots. 
\end{split}\]

Observe that this construction gives sequences that differ from the ones in \cite[Section 5.3]{MM16} in that it uses fewer repetitions of words and higher digits occur earlier.
\end{ex}

\begin{ex}\label{voorbeeld2}
We now let $L= (\frac{1}{2^d})_{d \in \mathbb{N}},$ as in \cite[Proposition 4.6 (iii)]{DLR20}. Observe that all labels $p_n, q_n$ of the $L$-tree are equal to $\frac{1}{2}$. This leads to the following construction.

\medskip

\noindent \underline{Depth 1 and 2:} The labels of these depths are equal to those in the previous example, so we can choose again $a_1=a_2=a_3=1, a_4=2$ and $c(1)=1,c(2)=c(3)=2, c(4)=3$.

\medskip
\noindent \underline{Depth 3:} We again have to add 6 words from the set $G(3) = \{ 111, 12, 21, 3\}$. They now all have weights $\frac14$, which implies (for small enough $\varepsilon$) that we add all words at least once and two words of choice twice. We choose to add $111$ and $12$ twice. As we again still do not have to take (P3) into account, we can choose e.g.~$\bm{u}_4 = 111$, $\bm{u}_5 = 111$, $\bm{u}_6=12$, $\bm{u}_7 = 12$, $\bm{u}_8 = 21$, $\bm{u}_9=3$. This gives
\[ a_5 \cdots a_{17} = 111\, 111\, 12\, 12\, 21\, 3.\]
Hence, $c(5)=c(6)=c(7)=4$, $c(8)=c(9)=c(10)=5$, $c(11)=c(12)=6$, $c(13)=c(14)=7$, $c(15)=c(16)=8$, $c(17)=9$.

\medskip
\noindent \underline{Depth 4:} To conclude this example, we add 24 words from $G(4) = \{ 1111, 112, 121, 13, 211, 22, 31, 4\}$. All weights are $\frac18$ so we add every word 3 times. Taking (P3) into account, we alternate between the first four words $\{ 1111, 112, 121, 13\}$ and the last four words $\{211, 22, 31,4\}$. An example is
\[ \begin{array}{cc|cc|cc|cc|cc|cc}
\bm{u}_{10} & \bm{u}_{11} & \bm{u}_{12} & \bm{u}_{13} & \bm{u}_{14} & \bm{u}_{14} & \bm{u}_{16} & \bm{u}_{17} & \bm{u}_{18} &
\bm{u}_{19} & \bm{u}_{20} & \bm{u}_{21}\\
1111 & 211 & 1111 & 211 & 1111 & 211 & 112 & 22 & 112 & 22 & 112 & 22\\
%a_{18}a_{19}a_{20}a_{21} & a_{22} a_{23} a_{24} & a_{25} a_{26} a_{27} & a_{28} & a_{29} a_{30} & a_{31} a_{32} a_{33} a_{34} & a_{35} a_{36} & a_{37} a_{38} a_{39} & a_{40} a_{41} a_{42} & a_{43} a_{44} & a_{45} a_{46}\\
\\
\bm{u}_{22} & \bm{u}_{23} & \bm{u}_{24} &
\bm{u}_{25} & \bm{u}_{26} & \bm{u}_{27} &
\bm{u}_{28} & \bm{u}_{29} & \bm{u}_{30} &
\bm{u}_{31} & \bm{u}_{32} & \bm{u}_{33}\\
121 & 31 & 121 & 31 & 121 & 31 & 13 & 4 & 13 & 4 & 13 & 4,
\end{array}\]
which so far yields
\[ \begin{split}
A =\ & 1 \, 11\, 2 \, 111\, 111\, 12\, 12\, 21\, 3 \, 1111\, 211\, 1111 \, 211\, 1111\, 211\, 112\, 22\, 112\, 22\, 112 \, 22\, 121\, 31\, 121\, 31\, 121\, 31 \\ 
& 13\, 4\, 13\, 4\, 13\, 4 \cdots. 
\end{split}\]

Note that regardless of our choices this procedure will always yield different sequences compared to the sequence $\mathcal{K}$ constructed in \cite{DLR20}, where higher digits occur much earlier.

\end{ex}

\section{Normal numbers}\label{s:GLS}

In this section we project the $L$-normal sequences obtained in the previous section to $[0,1]^n$ for some $n \in \mathbb N$ to obtain normal numbers in certain number systems.

\subsection{GLS expansions}\label{ss:gls}
GLS expansions are number representations of real numbers in $[0,1]$. We will introduce GLS expansions through the algorithm that produces them. Let $\mathcal I = \{ (\ell_d, r_d] \, : \, d \in \mathbb N\}$ be a countable collection of pairwise disjoint subintervals of $[0,1]$, set $L_d = r_d - \ell_d$, $d \in \mathbb N$, for the lengths of these intervals and assume that $\mathcal I$ satisfies $0< L_{d+1} \le L_d < 1$ for each $d \in \mathbb N$ and $\sum_{d \in \mathbb N} L_d =1$. In other words, $\mathcal I$ gives an interval partition of $[0,1]$ and the intervals are ordered in size. Let $(\varepsilon_d)_{d \in \mathbb N} \in \{0,1\}^\mathbb N$. Then the {\em GLS transformation} associated to $\mathcal I$ and $(\varepsilon_d)_{d \in \mathbb N}$ is the map  $T:[0,1] \to [0,1]$ given by
\begin{equation}\label{q:glstrf} T(x) = \begin{cases}
 \frac{(-1)^{\varepsilon_d}(x-\ell_d) + \varepsilon_d L_d}{L_d}, & \text{if } x \in (\ell_d, r_d], \, d \in \mathbb N,\\
0, & \text{if } x \not \in \bigcup_{d \ge 1} (\ell_d, r_d],
\end{cases}\end{equation}
see Figure~\ref{f:gls}. For $x \in [0,1]$ for which $T^i(x) \in \bigcup_{d \in \mathbb N} (\ell_d, r_d]$ for each $i \ge 0$ we can define two sequences by setting for each $i \in \mathbb N$ the $i$-th {\em digit} and {\em sign}, respectively, by $a_i =a_i(x) = d$ and $s_i=s_i(x) = \varepsilon_d$ if $T^{i-1}(x) \in (\ell_d, r_d]$, $d \in \mathbb N$. Then iterations of $T$ give a GLS expansion of the form
\[ x = \sum_{n \ge 1} (-1)^{\sum_{i=1}^{n-1}s_i} (\ell_{a_n}+s_nL_{a_n}) \prod_{i=1}^{n-1}L_{a_i},\]
where we set $(-1)^{\sum_{i=1}^0s_i}=1$ and $\prod_{i=1}^0L_{a_i}=1$. (Note that there are at most countably many points $x \in [0,1]$ that do not satisfy $T^i(x) \in \bigcup_{d \in \mathbb N} (\ell_d, r_d]$ for each $i \ge 0$.) The L\"uroth expansions and the L\"uroth transformation $T_L$ are recovered by setting $r_1=1$ and $r_{n+1}=\ell_n = \frac1{n+1}$ and $\varepsilon_n=0$ for $n \in \mathbb N$.

\medskip
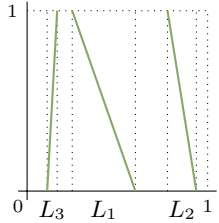
\begin{figure}[h]
\begin{tikzpicture}[scale=2.4]
\draw(-.05,0)node[below]{\scriptsize 0}--(15/108,0)node[below]{\small $L_3$}--(17/40,0)node[below]{\small $L_1$}--(247/288,0)node[below]{\small $L_2$}--(1,0)node[below]{\scriptsize 1}--(1.05,0)(0,-.05)--(0,1)node[left]{\scriptsize 1}--(0,1.05);
\draw[dotted](1/9,0)--(1/9,1)(1/6,0)--(1/6,1)(1/4,0)--(1/4,1)(3/5,0)--(3/5,1)(7/9,0)--(7/9,1)(15/16,0)--(15/16,1)(0,1)--(1,1)--(1,0);
\draw[thick, asparagus](1/9,0)--(1/6,1)(1/4,1)--(3/5,0)(7/9,1)--(15/16,0);
\end{tikzpicture}
\caption{The graph of a GLS transformation restricted to $(\ell_1, r_1] \cup (\ell_2,r_2] \cup (\ell_3,r_3]$ with $\varepsilon_1=\varepsilon_2=1$ and $\varepsilon_3=0$.}
\label{f:gls}
\end{figure}

\medskip
Since $a_n(x) = d$ if $T^{n-1}(x) \in (\ell_d, r_d]$, we expect the digit $d$ to occur with frequency $L_d$. In fact, it is a consequence of the Birkhoff ergodic theorem, see \cite{BBDK96}, that this is indeed the case for Lebesgue almost all $x \in [0,1]$. Therefore, we call an $x \in [0,1]$ {\em GLS normal} with respect to $\mathcal I$ and  $(\varepsilon_d)_{d \in \mathbb N}$ %the maps $T$ are invariant and ergodic with respect to the Lebesgue measure and hence, again by Birkhoff's ergodic theorem, for Lebesgue almost all $x \in [0,1]$ it holds that
if for any $k \ge 1$ and any digits $\alpha_1, \ldots, \alpha_k \in \mathbb N$ it holds for the digit sequence $(a_i)_{i \in \mathbb N}$ of $x$ that
\[ \lim_{n \to \infty} \frac{\# \{ 1 \le i \le n \, : \, a_i = \alpha_1, \ldots, a_{i+k-1}=\alpha_k\}}{n} = \prod_{j=1}^n L_{\alpha_j},\]
or equivalently, if the digit sequence $(a_i)_{i \in \mathbb N}$ is an $(L_d)_{d \in \mathbb N}$-normal sequence. Note that the signs $\varepsilon_d$ have no effect on the length of the interval $(\ell_d,r_d]$ and thus also no effect on the value of $L_d$. Therefore, if we let $A = (a_i)_{i \in \mathbb N}$ be any $(L,K_1,K_2)$-tree sequence and set the sequence of signs $(s_i)_{i \in \mathbb N}$ to be the one corresponding to $(\varepsilon_d)_{d \in \mathbb N}$, so $s_i = \varepsilon_{a_i}$ for each $i \in \mathbb N$, then the number
\[ x =  \sum_{n \ge 1} (-1)^{\sum_{i=1}^{n-1}s_i} (\ell_{a_n}+s_nL_{a_n})  \prod_{i=1}^{n-1}L_{a_i}\]
is GLS normal.

%\medskip
%Another way of formulating GLS normality is the following. For each $ \alpha_1, \ldots, \alpha_k \in \mathbb N$ let
%\[ I_{\alpha_1, \ldots, \alpha_k}:= (\ell_{\alpha_1}, r_{\alpha_1}] \cap T^{-1} (\ell_{\alpha_2}, r_{\alpha_2}] \cap \cdots \cap T^{-(k-1)} (\ell_{\alpha_k}, r_{\alpha_k}]\]
%be the set of points $y \in (\ell_{\alpha_1}, r_{\alpha_1}]$ that satisfy $T^j(y) \in (\ell_{\alpha_{j+1}}, r_{\alpha_{j+1}}]$ for all $1 \le j <k$. Then for each $ \alpha_1, \ldots, \alpha_k \in \mathbb N$ under iterations of $T$ the proportion of time that the GLS normal point $x$ spends in the set $I_{\alpha_1, \ldots, \alpha_k}$ is equal to the Lebesgue measure of $I_{\alpha_1, \ldots, \alpha_k}$.

\begin{ex}
We calculate the GLS normal numbers corresponding to Example~\ref{voorbeeld1} and Example~\ref{voorbeeld2}. For simplicity, for both examples we take $\varepsilon_d=0$ for all $d \in \mathbb{N},$ and let $\mathcal{I}$ be such that the intervals are ordered from right to left in size.

\medskip
For Example~\ref{voorbeeld1} this gives $\mathcal{I} = \{(\frac1{d+1}, \frac1{d}] \, :\,  d \in \mathbb{N}\}$, so that the projection of the sequence $A = 1112111111122133\cdots $ from Example~\ref{voorbeeld1} will yield a L\"uroth normal number. We obtain $\ell_1 = \frac12$, $\ell_2 = \frac13$, $L_1 = \frac12$ and $L_2=\frac16$. Therefore, the L\"uroth normal number given by $A$ is
\begin{equation*}
        \begin{split}
            x &= \ell_1  + L_1\ell_1 + L_1^2 \ell_1 + L_1^3 \ell_2 + L_1^3L_2 \ell_1 + \cdots\\
            &= \frac{1}{2} + \frac{1}{2^2} + \frac{1}{2^3} + \frac{1}{2^3\cdot 3} + \frac1{2^3 \cdot 6 \cdot 2} + \dots = 0.9374 \dots ,\\
        \end{split}
    \end{equation*}
where the last string of digits represents the decimal expansion of $x$.

\medskip
For Example~\ref{voorbeeld2} we obtain  $\mathcal{I} = \{ (\frac{1}{2^d}, \frac{1}{2^{d-1}}] : d \in \mathbb{N}\}$. We now have $\ell_1 = \frac12$, $\ell_2 = \frac14$, $L_1 = \frac12$ and $L_2=\frac14$, which for the sequence $A=11121111111212213 \cdots$ gives 
\begin{equation*}
        \begin{split}
            x &= \ell_1  + L_1\ell_1 + L_1^2 \ell_1 + L_1^3 \ell_2 + L_1^3L_2 \ell_1 + \cdots\\
            &= \frac{1}{2} + \frac{1}{2^2} + \frac{1}{2^3} + \frac{1}{2^3\cdot 4} + \frac1{2^3 \cdot 4 \cdot 2} + \dots  = 0.9373 \dots .\\
        \end{split}
    \end{equation*}
\end{ex}
\subsection{Multidimensional GLS expansions}
We can extend the number system from the previous section to higher dimensions. Fix some $N \in \mathbb N$ and for each $1 \le k \le N$ let $\mathcal I^{k} = \{ (\ell_d^{(k)}, r_d^{(k)}] \, : \, d \in \mathbb N\}$ and $(\varepsilon_d^{(k)})_{d \in \mathbb N}$ be as in Section~\ref{ss:gls}. For each $1 \le k \le N$ let $T_k$ be the GLS transformation for $\mathcal I^{k}$ and $(\varepsilon_d^{(k)})_{d \in \mathbb N}$ as given in \eqref{q:glstrf}. We can define the map $\hat T: [0,1]^N \to [0,1]^N$ by setting
\[ \hat T (\bm{x})  = \hat T(x_1, \ldots, x_N) = (T_k(x_k))_{1 \le k \le N}.\]
Then applications of $\hat T$ simultaneously produce GLS expansions for the coordinates $x_k$ by iteration in the sense that for points $\bm{x} = (x_1, \ldots, x_N)$ such that $T^n_k(x_k) \in \bigcup_{d \in \mathbb N} (\ell_d^{(k)},r_d^{(k)}]$ for each $1 \le k \le N$ iterations of $\hat T$ assign to each $x_k$ a digit sequence $(a_i^{(k)})_{k \ge 1}$ and a sign sequence $(s_i^{(k)})_{k \ge 1}$ so that the point $\bm{x}$ is given by
\[ \bm{x} = \left( \sum_{n \ge 1} (-1)^{\sum_{i=1}^{n-1}s_i^{(k)}} (\ell_{a_n^{(k)}}+s_n^{(k)}L_{a_n^{(k)}})  \prod_{i=1}^{n-1}L_{a_i^{(k)}} \right)_{1 \le k \le N}.\]
Therefore, $\bm{x}$ is represented by the digit sequence $(a_i^{(1)}, \ldots, a_i^{(N)})_{i \in \mathbb N} \in (\mathbb N^N)^\mathbb N$ and the sign sequence $(s_i^{(1)}, \ldots, s_i^{(N)})_{i \ge 1} \in (\{0,1\}^N)^\mathbb N$. Assume there exists a bijection $f: \mathbb N^N \to \mathbb N$ that satisfies the following: if for $\bm{i}= (i_1, \ldots, i_N) , \bm{j} = (j_1, \ldots, j_N) \in \mathbb N^N$ it holds that
\[ \prod_{k=1}^N L^{(k)}_{i_k} \ge \prod_{k=1}^N L^{(k)}_{j_k},\]
then $f(\bm{i}) \le f(\bm{j})$. Then let $L = (L_d)_{d \in \mathbb N}$ be the positive probability sequence with $L_d = \prod_{k=1}^N L^{(k)}_{d_k}$, where $(d_1, \ldots, d_N) = f^{-1}(d)$.

\medskip
The map $\hat T$ is invariant and ergodic with respect to the $N$-dimensional Lebesgue measure on $[0,1]^N$. Therefore, the Birkhoff ergodic theorem implies that for each $\alpha_1, \ldots, \alpha_k \in \mathbb N$ and Lebesgue almost all $\bm{x} \in [0,1]^N$ we have 
\begin{equation}\label{q:multidimgls}
 \lim_{n \to \infty} \frac{\# \{ 1 \le i \le n \, : \, f(a_i^{(1)}, \ldots, a_i^{(N)}) = \alpha_1, \ldots, f(a_{i+k-1}^{(1)}, \ldots, a_{i+k-1}^{(N)}) =\alpha_k\}}{n} = \prod_{j=1}^k L_{\alpha_j}.
 \end{equation}
Hence, it would be natural to call an $\bm{x} \in [0,1]^N$ normal in this multidimensional GLS number system if \eqref{q:multidimgls} holds for all $\alpha_1, \ldots, \alpha_k \in \mathbb N$, $k \ge 1$. Let $A = (a_i)_{i \in \mathbb N}$ be any $(L,K_1,K_2)$-tree sequence. For each $i \in \mathbb N$ let $(a_i^{(1)}, \ldots, a_i^{(N)}) = f^{-1}(a_i)$ and let the sign sequence $(s_i^{(1)}, \ldots, s_i^{(N)})_{i \in \mathbb N}$ be given by $s_i^{(k)} = \varepsilon_{a_i^{(k)}}^{(k)}$, $1 \le k \le N$, $i \ge 1$. Then the point
\[ \bm{x} = \left( \sum_{n \ge 1} (-1)^{\sum_{i=1}^{n-1}s_i^{(k)}} (\ell_{a_n^{(k)}}+s_n^{(k)}L_{a_n^{(k)}})  \prod_{i=1}^{n-1}L_{a_i^{(k)}} \right)_{1 \le k \le N}\]
is normal according to this definition.

\def\cprime{$'$}

\end{document}